\documentclass{amsproc}
\usepackage{amsmath,amsthm,amsfonts,amscd,amssymb,amstext}
\usepackage{mathtools}
\usepackage{physics}

\numberwithin{equation}{section}
\usepackage{booktabs}
\usepackage{esint}
\usepackage{relsize}
\usepackage{graphicx,color}
\usepackage{tabularx}
\usepackage{titletoc}
\usepackage{microtype}
\usepackage{caption}
\usepackage{subcaption}
\usepackage[colorlinks=true,linkcolor=black,citecolor=black,urlcolor=blue]{hyperref}
\usepackage{listings}
\usepackage{cite}
\usepackage{url}
\usepackage{alltt}
\usepackage{tikz}
\usepackage{bbm}
\usetikzlibrary{arrows,cd,decorations.markings,positioning}
\usepackage{tkz-graph}
\tikzstyle{vecArrow} = [thick, decoration={markings,mark=at position
   1 with {\arrow[semithick]{open triangle 60}}},
   double distance=1.4pt, shorten >= 5.5pt,
   preaction = {decorate},
   postaction = {draw,line width=1.4pt, white,shorten >= 4.5pt}]
\tikzstyle{innerWhite} = [semithick, white,line width=1.4pt, shorten >= 4.5pt]

\theoremstyle{plain}
\newtheorem{theorem}{Theorem}
\newtheorem*{theorem*}{Theorem}
\numberwithin{theorem}{section}

\newtheorem{lemma}[theorem]{Lemma}
\newtheorem{proposition}[theorem]{Proposition}
\newtheorem{cor}[theorem]{Corollary}

\newtheorem{definition}{Definition}

\theoremstyle{remark}
\newtheorem*{remark}{Remark}

\newtheorem*{fom}{Fomin's identity}
\newtheorem*{aff}{An affine version of Fomin's identity}

\newcommand{\barr}{\begin{eqnarray}}
\newcommand{\earr}{\end{eqnarray}}
\newcommand{\be}{\begin{equation}}
\newcommand{\ee}{\end{equation}}

\begin{document}

\title[]
{Loop-erased walks and random matrices}

\author[]{Jonas Arista}
\address{\newline School of Mathematics and Statistics, University College Dublin, Dublin 4, Ireland}

\author[]{Neil O'Connell}

\date{\today}
\begin{abstract}
It is well known that there are close connections between 
non-intersecting processes in one dimension
and random matrices, based on the reflection principle.  There is a generalisation
of the reflection principle for more general (e.g. planar) processes, due to S. Fomin, in which the 
non-intersection condition is replaced by a condition involving loop-erased paths.  
In the context of independent Brownian motions in suitable planar domains, this 
also has close connections to random matrices.  An example of this was first observed
by Sato and Katori (Phys. Rev. E, 83, 2011).  We present further examples
which give rise to various Cauchy-type ensembles.  We also extend Fomin's identity to
the affine setting and show that in this case, by considering independent Brownian 
motions in an annulus, one obtains a novel interpretation of the circular orthogonal ensemble.
\end{abstract}



\maketitle


\section{Introduction}
\label{sec:intro}

It is well known that there are close connections between 
non-intersecting processes in one dimension
and random matrices, based on the reflection principle. 
There is a generalisation of the reflection principle for more general processes, 
due to S. Fomin~\cite{Fomin}, in which the non-intersection condition is replaced 
by one involving loop-erased paths.  
In the context of independent Brownian motions in suitable planar domains, this 
also has close connections to random matrices, specifically Cauchy-type ensembles.
An example of this was first observed by Sato and Katori \cite{Katori}.  
We will present further examples, in particular, based on some domains
which were discussed in Fomin's original paper.
We will also consider the circular setting, with periodic boundary conditions,
for this we extend Fomin's identity to the affine setting; we show that in this case, 
by considering independent Brownian motions in an annulus, we obtain a 
novel interpretation of the Circular Orthogonal Ensemble of random matrix theory.

\subsection{Determinant formulas for loop-erased walks and affine generalisations}
\label{dflew}
Determinant formulas for the total weight of one-dimensional non-intersecting processes have many variations, both in continuous and discrete settings.
They are also known as the Karlin-McGregor formula for Markov processes \cite{Karlin, Werner, Grabiner1,Neil1}, or the Lindstr\"om-Gessel-Viennot lemma in enumerative combinatorics
\cite{Lindstrom,Gessel-Viennot,Stembridge,Gessel}. Roughly speaking, the argument behind all these determinant formulas is the classical reflection principle, which allows the construction of a particular one-to-one `path-switching' map from a set of {\it intersecting} paths onto itself, such that the map is its own inverse 
(see Section \ref{rp}).

For {\it two-dimensional} state space processes, it is not clear how to perform the classical reflection principle, since the paths under consideration are allowed to have self-intersections (or  loops). However, there is a generalisation of the reflection principle for more general (e.g. planar) paths, due to S. Fomin \cite{Fomin}, in which the non-intersecting condition is replaced by one involving {\it loop-erased paths}. Then it is possible to obtain a determinant formula (Theorem \ref{Fomingen}) for the total weight of discrete planar processes which satisfy Fomin's non-intersection condition, 
here stated in the context of Markov chains:

\begin{fom}
Consider a time-homogeneous Markov chain whose state space is a discrete subset $V$ of a simply connected domain $\Omega$. Assume that the transitions of the chain are determined by a (weighted) planar directed graph (with vertex set $V$). Multiple loops are allowed. Distinguish a subset $\partial\Gamma\subset V$ of {\it boundary vertices} and assume they all lie on the topological boundary $\partial\Omega$. Assume that vertices $a_{n},...,a_{1}\subset V$ and $b_{1},...,b_{n}\subset\partial\Gamma$ lie on the boundary $\partial\Omega$ and are ordered counterclockwise (along $\partial\Omega$), as in Figure \ref{planardomain}. Therefore, if 
$$h(a_{i},b_{j})\quad 1\leq i,j\leq n,$$ 
denotes the probability (or {\it hitting probability}) that the Markov chain, starting at $a_{i}$, will first hit the boundary $\partial\Gamma$ at vertex $b_{j}$ (if $a_{i}\in\partial\Gamma$, the chain is supposed to walk into $V\setminus\partial\Gamma$ before reaching $b_{j}$), then the $n\times n$ determinant
\begin{align}\label{Fomingeneral11}
\det(h(a_{i},b_{j}))_{i,j=1}^{n},
\end{align}
is equal to the probability that $n$ independent trajectories of the Markov chain $X_{1},...,X_{n}$, starting at $a_{1},...,a_{n}$, respectively, will first hit the boundary $\partial\Gamma$ at locations $b_{1},...,b_{n}$, respectively, and furthermore the trajectory $X_{j}$ will never intersect the loop-erasure $LE(X_{i})$ of $X_{i}$, for all $i<j$, that is, 
\begin{align}\label{repuls}
X_{j}\cap LE(X_{i})=\emptyset,\quad\text{for all}\,\,1\leq i<j\leq n.
\end{align}
\end{fom}
\vspace{0.1cm}
The above identity is the non-acyclic analogue of the determinant formula for non-intersecting one-dimensional processes of  Karlin-McGregor/Gessel-Viennot. In this respect, the following details are worth to remark: because of the nature of the underlying graph, trajectories of the Markov chain are allowed to have loops and therefore, for a given trajectory, we can properly define its {\it loop-erasure} as the self-avoiding path resulting from erasing its loops chronologically. Moreover, the determinant (\ref{Fomingeneral11}) gives the locations of the hitting points $b_{1},...,b_{n}$ along the boundary $\partial\Gamma$, and the condition on the trajectories is given by (\ref{repuls}), which forces the loop-erased paths to repel each other (see Section \ref{rplerw}). The counterclockwise arrangement of paths is just a particular case in the more general combinatorial identity given by S. Fomin in \cite{Fomin}, which can be applied to a wide range of configurations of $n$ distinct paths, depending on the location of the initial and final vertices and the topology of the planar domain $\Omega$.

Section \ref{avfi} is a first step towards the extension of the previous framework to non-simply connected domains of the complex plane. There, we state and prove an affine (circular) version of Fomin's identity (Proposition \ref{Fominaff3}), which can be seen as an extension of Fomin's identity to the setting of the affine symmetric group $\tilde{A}_{n}$. In Section \ref{circular} we relate this affine version with the Circular Orthogonal Ensemble (COE) of random matrix theory. In the context of Markov chains, our affine version of Fomin's identity can be stated as follows:

\begin{aff}
Consider a time-homogeneous Markov chain whose transitions are determined by the (directed) lattice strip $G=\mathbb{Z}\times\{0,1,...,N\}$. Assume that the transition probabilities are space-invariant with respect to a fixed horizontal translation $\mathcal{S}:G\to G$. If vertices $a_{n},...,a_{1},b_{1},...,b_{n}$ are ordered counterclockwise along the boundary (as in Figure \ref{peru}), then the $n\times n$ determinant
\begin{align}\label{Fomingeneral12}
\det\left(\sum_{k\in\mathbb{Z}}\zeta^{k}h(a_{i},\mathcal{S}^{k}b_{j})\right)_{i,j=1}^{n},
\end{align}
of hitting probabilities $h(a_{i},\mathcal{S}^{k}b_{j})$, where $\mathcal{S}^{k}=\mathcal{S}\circ\mathcal{S}^{k-1}$, $k\in\mathbb{Z}$, and 
$$
\zeta= \left\{
        \begin{array}{ll}
            1 & \quad \text{if $n$ is odd} \\
            -1 & \quad \text{if $n$ is even},
        \end{array}
    \right.
$$
is equal to the probability that $n$ independent trajectories of the Markov chain $X_{1},...,X_{n}$, starting at $a_{1},...,a_{n}$, respectively, will first hit the upper boundary $\partial\Gamma=\mathbb{Z}\times\{N\}$ at any of the $n$ cyclic permutations of the vertices $b_{1},...,b_{n}$, shifted also by all possible horizontal translations by $\mathcal{S}^{k}$, $k\in\mathbb{Z}$, and furthermore the trajectories are constrained to satisfy
\begin{align*}
P_{j}\cap LE(P_{j-1})=\emptyset,\quad 1< j\leq n,\quad\text{and}\quad P_{1}\cap LE(\mathcal{S}P_{n})=\emptyset.
\end{align*}
\end{aff}

It is important to note that the non-intersection condition above is related to the one between trajectories in a cylindrical lattice (or annulus on the complex plane), see Figure \ref{annulusband} and the introduction of Section \ref{avfi}. In an acyclic graph, the above affine case agrees with the Gessel-Zeilberger formula for counting paths in alcoves \cite{Gessel2} (see Section \ref{generallatt}).

\subsection{Scaling limits}\label{sclim} 

Our interest in the above determinant formulas relies upon their applicability in the context of suitable scaling limits of Fomin's identity and its affine version. It is well known that the two-dimensional Brownian motion $B$ is the scaling limit of simple random walks on different planar graphs \cite{Durrett2}. Moreover, the loop-erasure of those random walks converges (in a certain sense) to a random self-avoiding continuous path in the complex plane called SLE$(2)$, which belongs to the family of Schramm-Loewner evolutions, or SLE$(k)$, $k\geq0$, for short \cite{Lawler2,Schramm,Yadin}). As we might expect from our intuition, the latter SLE$(2)$ path is, in fact, a loop-erasure $LE(B)$ of the Brownian motion $B$ in a sense which can be made precise \cite{Zhan}. The previous considerations offer the possibility of interpreting, at least informally, the scaling limit of Fomin's identity and its affine version in terms of two-dimensional Brownian motions, in suitable complex domains. For example, since the determinants (\ref{Fomingeneral11}) and (\ref{Fomingeneral12}) involve hitting probabilities for a {\it single} Markov chain, it continues to make sense when $h(a,b)$ is the Poisson kernel (or hitting density) of two-dimensional Brownian motion in suitable simply connected domains $\Omega$ with smooth boundaries. One might expect that determinants of hitting densities are the scaling limits of the corresponding determinants of hitting probabilities for simple random walks, in square grid approximations of $\Omega$ and, moreover, that the former determinants express non-crossing probabilities between Brownian paths and SLE$(2)$ paths. This scaling limit has been rigorously achieved in the case of $n=2$ paths \cite{KL, KL3, KL2}, while ongoing works related to the general case $n>2$ are linked to the theory of (local and global) multiple SLE \cite{Dub1, KL, KL2, Karrila, Karrila2}.

\begin{figure}[tb]
\includegraphics[scale=0.9]{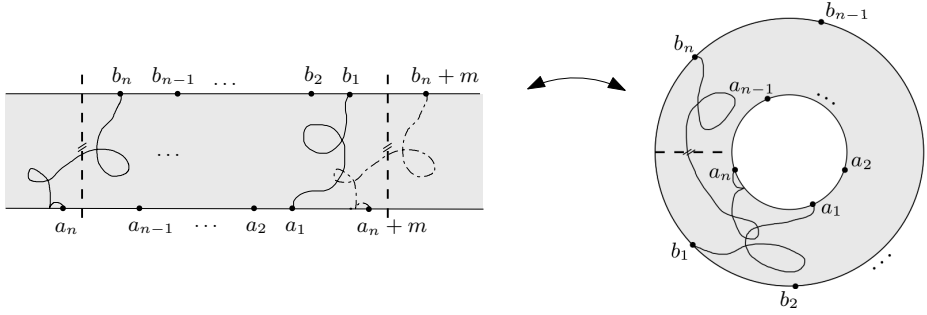}
\caption{Affine setting.}
\label{annulusband}
\end{figure}

Our contribution in the previous context is the connection with random matrix theory that emerges from the following setting: assume that $\Omega$ is a suitable complex (connected) domain with {\it smooth} boundary and $h(z_{0},y)$ is the (hitting) density of the harmonic measure
$$\mu_{z_{0},\Omega}(A)=\mathbb{P}^{z_{0}}(B_{T}\in A),\quad A\subset\partial\Omega,$$
with respect to one-dimensional Lebesgue measure (lenght), where $B$ under $\mathbb{P}^{z_{0}}$ denotes a two-dimensional Brownian motion starting at $z_{0}\in\Omega$, and $T=\inf\{t>0:B_{t}\notin\Omega\}$ is the first exit time of $\Omega$ (see Section \ref{abrociobm}). More generally, $h(z_{0},y)$ can be the hitting density of a diffusion in a suitable complex domain, with absorbing and normal reflecting boundary conditions (this idea is originally discussed in \cite{Fomin}). Therefore, for $m\in\mathbb{R}$ and appropriately chosen (parametrized) positions $x_{1},...,x_{n}$ and $y_{1},...,y_{n}$ along the boundary $\partial\Omega$, the {\it determinants of hitting densities}
\begin{align}\label{Poissonkernel4intro}
H(x,y)=\det\left(h(x_{i},y_{j})\right)_{i,j=1}^{n}dy_{1}\cdots dy_{n},
\end{align}
and
\begin{align}\label{Poissonkernel24intro}
H(x,y)=\det\left(\sum_{k\in\mathbb{Z}}\zeta^{k}h(x_{i},y_{j}+mk)\right)_{i,j=1}^{n}dy_{1}\cdots dy_{n},
\end{align}
where  
$$
\zeta= \left\{
        \begin{array}{ll}
            1 & \quad \text{if $n$ is odd} \\
            -1 & \quad \text{if $n$ is even},
        \end{array}
    \right.
$$
can be interpreted, informally, as the probability that $n$ independent `Brownian motions' $B_{i}$, $i=1,...,n$, starting at positions $x_{1},x_{2},...,x_{n}$, respectively, will first hit an absorbing boundary $\partial\Gamma\subset\partial\Omega$ at (parametrized) positions in the intervals $(y_{i}+dy_{i})$, $i=1,...,n$, and whose trajectories are constrained to satisfy the condition 
$$
B_{j}\cap LE(B_{i})=\emptyset,\quad\text{for all}\,\,1\leq i<j\leq n,
$$ 
in (\ref{Poissonkernel4intro}), or
\begin{align*}
B_{j}\cap LE(B_{j-1})=\emptyset,\quad 1< j\leq n,\quad\text{and}\quad B_{1}\cap LE(m+B_{n})=\emptyset,
\end{align*}
in the affine case (\ref{Poissonkernel24intro}). We remark that in the affine case, we assume $\Omega$ to be invariant under a fixed (horizontal) translation by $m\in\mathbb{R}$, and therefore $m+B_{n}$ is the horizontal translation by $m$ of the Brownian path $B_{n}$, see Figure \ref{annulusband}. We remark that some hitting densities $h(x,y)$ can be calculated explicitly for a number of important domains, like disks and half-planes, and many others can be deduced from these by reflection and conformal invariance of the two-dimensional Brownian motion. We consider examples of determinants of hitting densities in Sections \ref{rmt} and \ref{circular}.

Finally, if $\partial\Gamma=\partial\Omega$ and then the whole boundary $\partial\Omega$ is absorbing, we require a different notion of hitting density $h(x,y)$ (since the paths need to `walk' into the interior $\Omega^{\circ}=\Omega\setminus\partial\Omega$ before reaching their destination). Therefore, in order to study determinants of the form (\ref{Poissonkernel4intro}) and (\ref{Poissonkernel24intro}), we consider the so-called {\it excursion Poisson kernel}. In this context, an example of a similar interpretation of determinants of hitting densities of the form (\ref{Poissonkernel4intro}) was first observed by Sato and Katori \cite{Katori} (see Section \ref{nnoepkd}).

\subsection{Connections to random matrix theory}\label{crmt}
Non-intersecting processes in {\it one dimension} have long been an integral
part of random matrix theory, at least since the pioneering
work of Dyson~\cite{Dyson} in the 1960s.  For example, it is well known
that, if one considers $n$ independent one-dimensional Brownian particles, 
started at the origin
and conditioned not to intersect up to a fixed time $T$ (see Section \ref{an} for details), then the locations of the
particles at time $T$ 
have the same distribution as the eigenvalues of a random real symmetric 
$n\times n$ matrix with independent centered Gaussian entries, with 
variance $T$ on the diagonal and $T/2$ above the diagonal (this is known
as the Gaussian Orthogonal Ensemble (GOE)).  Similar statements
hold for the circular ensembles, see for example~\cite{Werner} or Proposition
\ref{jac} below.

In {\it two dimensions}, we can consider appropriate limits of the form
\begin{align}\label{normalisedlimitsintro}
\lim_{\substack{(x_{1},...,\,x_{n})\in C\\x_{i}\to x\in\partial\Omega}}\tilde{H}(x,y),\quad (y_{1},...,y_{n})\in C,
\end{align}
where $\tilde{H}(x,y)$ is an appropriate normalisation of the determinants $H(x,y)$ in (\ref{Poissonkernel4intro}) and (\ref{Poissonkernel24intro}), and the positions $x_{1},...,x_{n}$, $y_{1},...,y_{n}$ are determined by chambers (alcoves) $C$ of $\mathbb{R}^{n}$. These limits give the locations of the $n$ hitting points $y_{1},...,y_{n}$ along the absorbing boundary $\partial\Gamma$, when the processes start at a single common point $x\in\partial\Omega$. In a way, this is the two-dimensional analogue of the model described in the preceding paragraph.  In Section \ref{rmt} we show that the limits (\ref{normalisedlimitsintro}) agree with eigenvalue densities of Cauchy type random matrix ensembles, for determinants of the form (\ref{Poissonkernel4intro}) (see \cite{Katori} and Section \ref{nnoepkd} for similar asymptotic considerations regarding excursion Poisson kernels). For determinants of the form (\ref{Poissonkernel24intro}), in Section \ref{circular} we show that, by considering the hitting density of the two-dimensional Brownian motion in an annulus on the complex plane, certain limit of the form \ref{normalisedlimitsintro} agrees with the Circular Orthogonal Ensemble (COE) of random matrix theory (Proposition \ref{COE}).

\subsection{Organisation of the paper}
The paper is structured into two parts that can be read (essentially) independently. The first part (Sections \ref{background} and \ref{avfi}) is mainly concerned to the combinatorial results of Section \ref{dflew}.
In Section \ref{background} we give some background
on the reflection principle and Fomin's generalisation for loop-erased walks in discrete lattice models.
In Section \ref{avfi}, we present the affine version of Fomin's identity.
The second part (Sections \ref{rmt} and \ref{circular}) shows calculations and limits for determinants of hitting densities of the form (\ref{Poissonkernel4intro}) and (\ref{Poissonkernel24intro}).
In Section \ref{rmt}, we show that for suitable simply connected domains, the determinants associated with 
Fomin's identity converge, in a certain sense, to some known ensembles of 
random matrix theory.
In Section \ref{circular} we consider the affine setting and, after revisiting the model of non-intersecting one-dimensional Brownian motion on the circle \cite{Werner}, we show that a determinant of the form (\ref{Poissonkernel24intro}), in the context of independent Brownian motions in an annulus, 
converges in a suitable limit to the Circular Orthogonal Ensemble.

\bigskip

\noindent{\em Acknowledgements.}  
We gratefully acknowledge the support of the European Research Council 
(Grant number 669306) and CONACYT (PhD scholarship number 411059). 
We would also like to thank the anonymous referees for their careful reading
and suggestions, in particular for drawing our attention to the paper~\cite{Katori},
 which have led to a much improved version of the paper.

\section{The reflection principle and Fomin's generalisation}
\label{background}

In this section, we consider the discrete versions of some of the determinant formulas considered in the Introduction. This combinatorial approach has some advantages and will be particularly convenient in Section \ref{rplerw}, where some of the main concepts are defined for discrete paths. Let $G=(V,E,\omega)$ be a directed graph with no multiple edges, countable vertex set $V$ and edge set $E\subset V\times V$. The graph $G$ need not be acyclic, so multiple loops are allowed. The set $\omega$ is a family of pairwise distinct formal indeterminates $\{\omega(e)\}_{e\in E}$ that we will call the {\it weights} of the edges. The imposed restriction on edge multiplicity is not essential, but most of the applications we have in mind share this condition.

Let us introduce the notation and terminology we will use through all the following sections. A directed edge $e$ from vertex $a\in V$ to vertex $b\in V$ will be denoted as $a\stackrel{e}{\to}b$, and a {\it path} or {\it walk} $P$ will mean a finite sequence of (directed) edges and vertices 
$$P:a_{0}\stackrel{e_{1}}{\to}a_{1}\stackrel{e_{2}}{\to}a_{2}\stackrel{e_{3}}{\to}\cdots\stackrel{e_{n}}{\to}a_{n}.$$ 
In this case, we say that $P$ is a path from $a_{0}$ to $a_{n}$ of {\it length} $n$. For any pair of vertices $a,b\in V$, we denote the set of all paths in $G$ from $a$ to $b$ by $\mathcal{H}(a,b)$, and, if ${\bf a}=(a_{1},...,a_{n})$ and ${\bf b}=(b_{1},...,b_{n})$ are two $n$-tuples of vertices, then $\mathcal{H}({\bf a},{\bf b})$ will denote the set of $n$-tuples of paths  
\begin{align*}
\mathcal{H}({\bf a},{\bf b})=\{{\bf P}=(P_{1},...,P_{n}): P_{i}\in\mathcal{H}(a_{i},b_{i}),\,\,\,\text{for}\,\,\, 1\leq i\leq n\}. 
\end{align*}
The weight $\omega(P)$ of a path $P$ is defined as the product of its edge weights
\begin{align*}
\omega(P)=\prod_{i=1}^{n}\omega(e_{i}),
\end{align*}
if $P$ is given as above. Analogously, the weight of an $n$-tuple ${\bf P}=(P_{1},...,P_{n})$ is the product of the corresponding path weights $\omega({\bf P})=\prod_{i=1}^{n}\omega(P_{i})$. A quantity of interest will be the generating function 
\begin{align*}
h(a,b)=\sum_{P\in\mathcal{H}(a,b)}\omega(P),\quad a,b\in V,
\end{align*}
which encodes all paths $P\in\mathcal{H}(a,b)$ according to their weight. This expression should be understood as a formal power series in the independent variables $\{\omega(e)\}_{e\in E}$.

Finally, two paths $P_{1}$ and $P_{2}$ in $G$ {\it intersect} if they share at least one vertex (in their vertex-sequence definitions) and we will write this as $P_{1}\cap P_{2}\not=\emptyset$. A family of paths ${\bf P}\in\mathcal{H}({\bf a},{\bf b})$ is {\it intersecting} if any two of them intersect. We will say that $P$ is {\it self-avoiding} or has {\it no loops} if it does not visit the same vertex more than once, that is, if $a_{i}\not= a_{j}$ in the vertex sequence definition of $P$, for all $0\leq i<j\leq n$.

\subsection{The classical reflection principle}
\label{rp}

\begin{figure}[tb]
\includegraphics[scale=0.7]{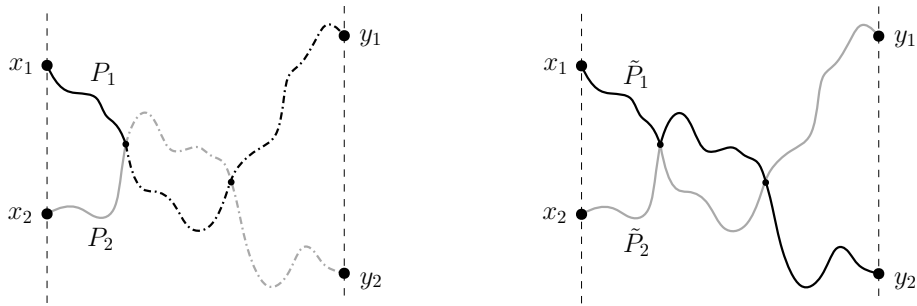}
\caption{The association $(P_1,P_2)\overset{\varphi}{\mapsto}(\tilde{P}_1,\tilde{P}_2)$ is an involution, $\varphi^{2}=id$.}
\label{rrpp}
\end{figure}

If the graph $G=(V,E,\omega)$ is acyclic (loops are not allowed), the reflection principle relies upon the following property. Consider two paths $P_{1}$ and $P_{2}$ in $G$, and assume that they intersect (see Figure \ref{rrpp}). Fix a total order for the set of vertices $V$ and let $A=\{v_{\alpha}:\alpha\in I\}$ be the set of intersection vertices between $P_{1}$ and $P_{2}$, which is finite. Among all intersection vertices, let $v_{\alpha_{0}}$ be the minimal with respect to the given order, and split the paths $P_{1}$ and $P_{2}$ at the vertex $v_{\alpha_{0}}$, into the corresponding subpaths:
\begin{align*}
P_{1}:a_{0}&\stackrel{P_{1}'}{\longrightarrow}v_{\alpha_{0}}\stackrel{P_{1}''}{\longrightarrow}a_{n}\\
P_{2}: a'_{0}&\stackrel{P_{2}'}{\longrightarrow}v_{\alpha_{0}}\stackrel{P_{2}''}{\longrightarrow}a'_{m}.
\end{align*}
Now interchange the parts $P_{1}''$ and $P_{2}''$ above. This procedure creates two new paths $\tilde{P}_{1}$ and $\tilde{P}_{2}$ given by
\begin{align*}
\tilde P_{1}:a_{0}&\stackrel{P_{1}'}{\longrightarrow}v_{\alpha_{0}}\stackrel{P_{2}''}{\longrightarrow}a'_{m}\\
\tilde P_{2}: a'_{0}&\stackrel{P_{2}'}{\longrightarrow}v_{\alpha_{0}}\stackrel{P_{1}''}{\longrightarrow}a_{n}.
\end{align*}
The paths $\tilde{P}_{1}$ and $\tilde{P}_{2}$ also intersect (in particular, $v_{\alpha_{0}}$ is an intersection vertex) and, more importantly, their set of intersection vertices is also $A=\{v_{\alpha}:\alpha\in I\}$. This means that the intersection vertices are {\it invariant} under the map $(P_{1},P_{2})\mapsto(\tilde{P}_{1},\tilde{P}_{2})$ and hence so is the minimum vertex $v_{\alpha_{0}}$.   Therefore, if we perform the same procedure to the paths $\tilde{P}_{1}$ and $\tilde{P}_{2}$, we recover the original paths $P_{1}$ and $P_{2}$. In other words, the map $(P_{1},P_{2})\mapsto(\tilde{P}_{1},\tilde{P}_{2})$ is an involution. Moreover, the
weights are also invariant under this operation: $\omega(P_{1})\omega(P_{2})=\omega(\tilde{P}_{1})\omega(\tilde{P}_{2})$.

A careful application of the above argument leads to the following enumeration formula for non-intersecting paths by Karlin and McGregor \cite{Karlin} (in the context of Markov chains) and Lindstr\"om \cite{Lindstrom} (further developed by Gessel-Viennot \cite{Gessel-Viennot}):

\begin{theorem}\label{kmlgv}
In an acyclic graph, let $\partial\Gamma\subset V$ be the distinguished  set of vertices:
$$\partial\Gamma=\{a\in V:\nexists\,a\stackrel{e}{\to}b\}.$$
For arbitrary sets $A=\{a_{1},...,a_{n}\}\subset V$ and $B=\{b_{1},...,b_{n}\}\subset\partial \Gamma$, it holds
\begin{align*}
\sum_{\sigma\in S_{n}}{\rm sgn}(\sigma)\sum_{\substack{{\bf P}\in\mathcal{H}({\bf a},{\bf b}_{\sigma})\\P_{i}\cap P_{j}=\emptyset,\,\,i\not=j}}\omega({\bf P})=\det\left(h(a_{i},b_{j})\right)_{i,j=1}^{n},
\end{align*}
where ${\bf b}_{\sigma}=(b_{\sigma(1)},...,b_{\sigma(n)})$.
\end{theorem}

\subsection{Loop-erased walks and Fomin's identity}
\label{rplerw}
If the graph $G=(V,E,\omega)$ is not acyclic, and the paths $P_{1}$ and $P_{2}$ intersect, then the invariance of the intersection vertices described in the previous section is no longer guaranteed, since an intersection vertex can be part of a loop.  
However, there is a modification of the reflection principle for general graphs, 
due to Fomin \cite{Fomin}, which we describe below.

We briefly present the key concept of {\it loop-erased walks} introduced by G. Lawler ~\cite{Lawler}.  
\begin{definition}\label{dlep}
For each path $P$ in $G=(V,E,\omega)$ of the form 
$$a_{0}\stackrel{e_{1}}{\to}a_{1}\stackrel{e_{2}}{\to}a_{2}\stackrel{e_{3}}{\to}\cdots\stackrel{e_{n}}{\to}a_{n},$$ 
the {\it loop-erasure} of $P$, denoted $LE(P)$, is the self-avoiding path obtained by chronological loop-erasure of $P$, as follows:
\begin{itemize}
	\item Let $j_{0}=\max\{j:a_{j}=a_{0}\}$;
	\item recursively, if $j_{k}<n$, then $j_{k+1}=\max\{j:a_{j}=a_{j_{k}+1}\}$;
	\item if $j_{k}=n$, then $LE(P)$ is the path $$a_{j_{0}}\stackrel{e_{j_{0}+1}}{\longrightarrow}a_{j_{1}}\stackrel{e_{j_{1}+1}}{\longrightarrow}a_{j_{2}}\stackrel{e_{j_{2}+1}}{\longrightarrow}\cdots\stackrel{e_{j_{k-1}+1}}{\longrightarrow}a_{j_{k}}.$$
\end{itemize}
\end{definition}
This procedure erases loops in $P$ in the order they appear, and the operation is iterated until no loop remains. In particular, note that $LE(P)$ is a subpath of the original path $P$, with the same starting and end points $a_{0}$ and $a_{n}$, respectively.

Using the above procedure, Fomin~\cite{Fomin} introduced the so-called {\it loop-erased switching} for paths that are allowed to self-intersect. The loop-erased switching is as follows: consider two paths $P_{1}:a_{0}=x_{1}\stackrel{}{\to}...\stackrel{}{\to}a_{n}=y_{2}$ and $P_{2}:a'_{0}=x_{2}\stackrel{}{\to}...\stackrel{}{\to}a'_{m}=y_{1}$ in the graph $G$, starting from different vertices $a_{0}\not=a'_{0}$, and assume $P_{2}$ and $LE(P_{1})$ intersect at least at one common vertex, that is, $P_{2}\cap LE(P_{1})\not=\emptyset$ (see Figure \ref{loopswitching}). Among all such intersection vertices, let $v=a_{j_{i}}$ be the one with minimal index along the vertex sequence of $LE(P_{1})$, (see Definition \ref{dlep}), and split the path $P_{1}$ at the end of the edge $a_{j_{i-1}}\stackrel{e_{j_{i-1}}}{\longrightarrow}v$ into two subpaths:

\begin{figure}[tb]
\includegraphics[scale=0.7]{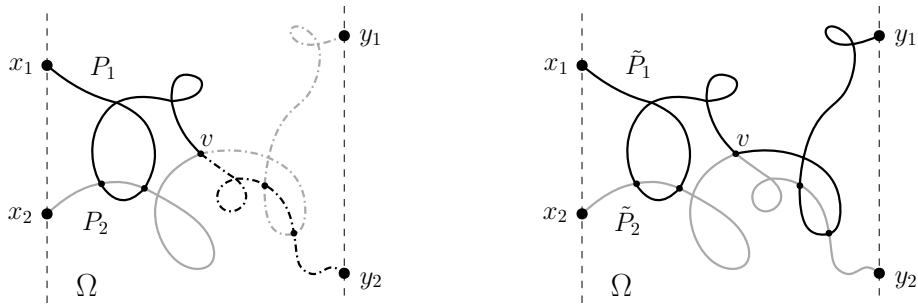}
\caption{The loop-erased switching. The path $P_{2}$ intersects the loop-erased part of $P_{1}$ and $v$ is the `first intersection' (left). Interchanging the paths at $v$, the new paths $\tilde{P}_{1}$ (black) and $\tilde{P}_{2}$ (gray) satisfy the same property, that is, $\tilde{P}_{2}$ intersects the loop-erased part of $\tilde{P}_{1}$ and $v$ is the `first intersection' (right).}
\label{loopswitching}
\end{figure}

\begin{align*}
P_{1}:\,a_{0}\stackrel{P_{1}'(v)}{\longrightarrow}v\stackrel{P_{1}''(v)}{\longrightarrow}a_{n}.
\end{align*}
This partition ensures that all possible loops of $P_{1}$ `rooted' at $v$ are part of $P_{1}''(v)$, so that $P_{1}''(v)$ does not intersect the path $LE(P_{1}'(v))$ at any vertex different from $v$. Now, if we split the path $P_{2}$ at its first visit to $v$, we have
\begin{align*}
P_{2}:\,a'_{0}\stackrel{P_{2}'(v)}{\longrightarrow}v\stackrel{P_{2}''(v)}{\longrightarrow}a'_{m}.
\end{align*} 
Then, by construction of $v$, $P_{2}''(v)$ does not visit any other vertex of $LE(P_{1}'(v))$, except for $v$, so it shares the same property as $P_{1}''(v)$. The latter common condition allows us to interchange the parts $P_{1}''(v)$ and $P_{2}''(v)$ at the vertex $v$, and create new paths 
\begin{align*}
\tilde P_{1}:a_{0}&\stackrel{P_{1}'(v)}{\longrightarrow}v\stackrel{P_{2}''(v)}{\longrightarrow}a'_{m}\\
\tilde P_{2}: a'_{0}&\stackrel{P_{2}'(v)}{\longrightarrow}v\stackrel{P_{1}''(v)}{\longrightarrow}a_{n}.
\end{align*}
Note that the new paths $\tilde{P}_{2}$ and $LE(\tilde{P}_{1})$ also intersect ($v$ is an intersection vertex), and therefore $\tilde P_{2}\cap LE(\tilde P_{1})\not=\emptyset$. These conditions ensure that the map $(P_{1},P_{2})\mapsto(\tilde P_{1},\tilde P_{2})$ is an involution, and the `minimality' of the intersection vertex $v$ is preserved, exactly as in Section \ref{rp}. We also have $\omega(P_{1})\omega(P_{2})=\omega(\tilde{P}_{1})\omega(\tilde{P}_{2})$.

The following theorem (Theorem 7.1 in \cite{Fomin}) is an application of the above loop-erased switching procedure. Fix a distinguished subset of vertices $\partial\Gamma\subset V$  and call it the {\it absorbing boundary}. For $a\in V$ and $b\in\partial\Gamma$, denote by $\mathcal{H}^{+}(a,b)\subset\mathcal{H}(a,b)$ the set of all paths of positive length
\begin{align*}
a\stackrel{e_{1}}{\to}a_{1}\stackrel{e_{2}}{\to}a_{2}\stackrel{e_{3}}{\to}\cdots\stackrel{e_{n}}{\to}b,
\end{align*}
such that all the internal vertices $a_{1},...,a_{n-1}$ lie in $V\setminus\partial\Gamma$. If $a\in\partial\Gamma$, we assume $n\geq2$, so that the path walks into $V\setminus\partial\Gamma$ before reaching the vertex $b$. Analogously, define $\mathcal{H}^{+}({\bf a},{\bf b})$ for $n$-tuples of paths ${\bf P}=(P_{1},...,P_{n})$ as at the beginning of Section \ref{background}, that is
\begin{align*}
\mathcal{H}^{+}({\bf a},{\bf b})=\{{\bf P}=(P_{1},...,P_{n}): P_{i}\in\mathcal{H}^{+}(a_{i},b_{i}),\,\,\,\text{for}\,\,\, 1\leq i\leq n\}. 
\end{align*}

\begin{theorem}[Fomin's identity]\label{Fomingen}
Let $G=(V,E,\omega)$ be a graph satisfying the above assumptions and $\partial\Gamma\subset V$. Let $A=\{a_{1},...,a_{n}\}\subset V$ and $B=\{b_{1},...,b_{n}\}\subset\partial\Gamma$ be two labelled sets of different vertices. Therefore
\begin{align}\label{Fomingeneral}
\sum_{\sigma\in S_{n}}{\rm sgn}(\sigma)\sum_{\substack{{\bf P}\in\mathcal{H}^{+}({\bf a},{\bf b}_{\sigma})\\P_{j}\cap LE(P_{i})=\emptyset,\,\,i<j}}\omega({\bf P})=\det(h(a_{i},b_{j}))_{i,j=1}^{n},
\end{align}
where ${\bf b}_{\sigma}=(b_{\sigma(1)},...,b_{\sigma(n)})$ and 
$$h(a,b)=\sum_{P\in\mathcal{H}^{+}(a,b)}\omega(P),\quad a\in V, b\in\partial\Gamma.$$
\end{theorem}

\begin{remark}
Note that the above theorem agrees with Theorem \ref{kmlgv} if the graph under consideration is acyclic. Also, note that Theorem \ref{Fomingen} does {\it not} give the total weight of families of non-intersecting paths in $G$ connecting $A$ and $B$ (in the strict sense of non-intersection). However, the paths are constrained to satisfy
\begin{align*}
P_{j}\cap LE(P_{i})=\emptyset,\quad\text{for all}\,\,\,i<j,
\end{align*} 
which forces the corresponding loop-erased parts to repeal each other.
\end{remark}

\begin{cor}\label{Fomin}
Assume that $G$ is planar and it is also embedded into a connected planar domain $\Omega$ in such a way that the vertices in the absorbing boundary $\partial\Gamma$ lie on the topological boundary $\partial\Omega$. Let $A\subset V$ and $B\subset\partial\Gamma$ be as in Theorem \ref{Fomingen}, and, whenever $i>i'$ and $j<j'$, assume that every path $P\in\mathcal{H}^{+}(a_{i},b_{j})$ intersects every path $P'\in\mathcal{H}^{+}(a_{i'},b_{j'})$ at a vertex in $V\setminus\partial\Gamma$ (see Figure \ref{planardomain}). In this case, the only allowable permutation in (\ref{Fomingeneral}) is the identity permutation, and therefore 
\begin{align}\label{Fomin1}
\sum_{\substack{{\bf P}\in\mathcal{H}^{+}({\bf a},{\bf b})\\P_{j}\cap LE(P_{j-1})=\emptyset,\,\,1<j\leq n}}\omega({\bf P})=\det(h(a_{i},b_{j}))_{i,j=1}^{n}.
\end{align}
In particular, if the weight function $\omega$ is non-negative, then the right hand side of (\ref{Fomin1}) is non-negative.
\end{cor}

\begin{figure}[tb]
\includegraphics[scale=0.5]{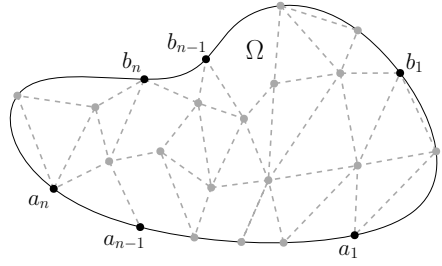}
\caption{The graph $G$ is embedded into $\Omega$ and the vertices $a_{n},...,a_{1},b_{1},...,b_{n}$ are ordered counterclockwise along $\partial\Omega$.}
\label{planardomain}
\end{figure}

\begin{remark}\label{remark1}
Assume that the vertex set $V$ is the state space of a time-homogeneous Markov chain $X$ and the possible transitions between states are determined by the planar graph $G$. That is, the transition probabilities $p(a,b)$ are positive if and only if there is an edge $a\stackrel{e}{\to}b$, in which case $\omega(e)=p(a,b)$.
Then the assertion of Corollary \ref{Fomin} has the following probabilistic interpretation: the generating function
\begin{align*}
h(a,b)=\sum_{P\in\mathcal{H}^{+}(a,b)}\omega(P),\quad a\in V, b\in\partial\Gamma,
\end{align*} 
is the {\it hitting probability} $\mathbb{P}_{a}(X_{T}=b, T<\infty)$, where $T$ is the first time the chain $X$ hits the boundary $\partial V$ (if $a\in\partial\Gamma$, the Markov chain is supposed to walk into $V\setminus\partial\Gamma$ before reaching $\partial\Gamma$). Then the left hand side of (\ref{Fomin1}) is equal to the probability that $n$ independent trajectories $X^{1},...,X^{n}$ of the Markov process $X$, starting at locations $a_{1},a_{2},...,a_{n}$, respectively, will hit the boundary $\partial\Gamma$ for the first time at the points $b_{1},b_{2},...,b_{n}$, respectively, and furthermore the trajectory $X^{j}$ will not intersect the loop-erased path $LE(X^{i})$ at any vertex in $V\setminus\partial\Gamma$, for all $i<j$, that is,
\begin{align*}
X^{j}\cap LE(X^{i})=\emptyset,\quad \text{for all}\,\,\,1\leq i<j\leq n.
\end{align*}

\end{remark}

\section{Affine version of Fomin's identity}
\label{avfi}

In this section, we extend Fomin's identity to the setting of the affine symmetric group (Theorem \ref{Fominaff}) and consider its natural projection onto the cylindrical lattice (Proposition \ref{Fominaff3}). Our main motivation is to present a preliminary extension of the framework considered in Section \ref{rmt} to non-simply connected domains, and show, in Section \ref{circular}, an interesting connection with circular ensembles of random matrix theory.

As we discussed in Section \ref{rplerw}, the interaction between $n$ paths $P_{1},P_{2},...,P_{n}$ imposed in Fomin's identity (Theorem \ref{Fomingen}) is given by the condition
\begin{align}\label{livre}
P_{j}\cap LE(P_{i})=\emptyset,\quad\text{for all}\,\,i<j.
\end{align} 
In particular, restricted to the lattice strip $G$ of Figure \ref{peru}, the above condition ensures a type of `repulsion' between consecutive paths from left to right, that is, every path $P_{j}$ will not intersect the loop-erased part $LE(P_{j-1})$ of the path to its right. Theorem \ref{Fominaff} below is an extension of Fomin's identity in the sense that we consider families of paths $P_{1},P_{2},...,P_{n}$ subject to (\ref{livre}) and also subject to an extra non-intersection condition between the path $P_{1}$ and the translation to the right of $P_{n}$, given by a fixed translation $\mathcal{S}$ of the graph $G$ (see Figure \ref{croacia}), that is
\begin{align}\label{livre2}
P_{1}\cap LE(\mathcal{S}P_{n})=\emptyset.
\end{align}
This type of interaction is helpful when the lattice strip $G$ is projected onto the {\it cylindrical lattice} $\tilde G$, modulo the translation $\mathcal{S}$ (or affine setting, see Sections \ref{projections} and \ref{lewcyl}). In this case, the conditions (\ref{livre}) and (\ref{livre2}) jointly ensure that the projected paths $\tilde{P}_{1},\tilde{P}_{2},...,\tilde{P}_{n}$, in $\tilde{G}$, also satisfy the analogous `left to right' non-intersection condition
\begin{align*}
\tilde{P}_{j}\cap LE(\tilde{P}_{j-1})=\emptyset,\quad 1< j\leq n,\quad\text{and}\quad\tilde{P}_{1}\cap LE(\tilde{P}_{n})=\emptyset.
\end{align*}

\subsection{Affine version of Fomin's identity}
\label{affinefomin}

Consider the lattice strip $G=(V,E,\omega)$, given by the vertex set $V=\mathbb{Z}\times\{0,1,...,N\}$ and connected by directed horizontal and vertical edges, in both positive and negative directions. We also assume that the weights $\{\omega(e)\}_{e\in E}$ are invariant under  horizontal translations by $v=(M,0)$, for some positive $M\in\mathbb{Z}$. Let $\partial\Gamma=\{(i,N): i\in\mathbb{Z}\}$ denote the upper boundary of the lattice strip and consider the set $\mathcal{H}^{+}({\bf a},{\bf b})$ for ${\bf a}=(a_{1},...,a_{n})$, ${\bf b}=(b_{1},...,b_{n})$ vectors of vertices, as defined in Section \ref{rplerw}. We have the following.

\begin{theorem}\label{Fominaff}
Consider integers $i_{n}<i_{n-1}<...<i_{1}<i_{n}+M$ and $j_{n}<j_{n-1}<...<j_{1}<j_{n}+M$. Define the following two $n$-tuples of vertices in $G$ 
\begin{align*}
a_{k}:=(i_{k},0),\quad\quad b_{k}:=(j_{k},N),\quad 1\leq k\leq n.
\end{align*} 
If $\mathcal{S}:G\to G$ is the horizontal translation by $(M,0)$, the $2(n+1)$ vertices $a_{n},...,a_{1},\mathcal{S}a_{n}$, $\mathcal{S}b_{n},b_{1},...,b_{n}$ are ordered counterclockwise along the topological boundary of the lattice strip $G$ (see Figure \ref{peru}). Therefore
\begin{align}\label{Fomin2}
\sum_{\substack{{\bf P}\in\mathcal{H}^{+}({\bf a},{\bf b}) \\ P_{j}\cap\text{LE}(P_{j-1})=\emptyset,\quad 1< j\leq n \\ P_{1}\cap\text{LE}(\mathcal{S}P_{n})=\emptyset}}\omega({\bf P})=\sum_{\sigma\in S_{n}}\sum_{\substack{k_{i}\in\mathbb{Z}\\k_{1}+k_{2}+...+k_{n}=0}}{\rm sgn}(\sigma)\sum_{{\bf P}\in\mathcal{H}^{+}({\bf a},{\bf \mathcal{S}^{k}b_{\sigma}})}\omega({\bf P}),
\end{align}
where ${\bf \mathcal{S}^{k}b_{\sigma}}=(\mathcal{S}^{k_{1}}b_{\sigma(1)},...,\mathcal{S}^{k_{n}}b_{\sigma(n)})$ and $\mathcal{S}^{k_{i}}=\mathcal{S}\circ\mathcal{S}^{k_{i}-1}$, $k_{i}\geq2$. If, as before, $h(a,b)=\sum_{P\in\mathcal{H}^{+}(a,b)}\omega(P)$, then the right hand side of (\ref{Fomin2}) takes the form
\begin{align*}
\sum_{\sigma\in S_{n}}\sum_{\substack{k_{i}\in\mathbb{Z}\\k_{1}+k_{2}+...+k_{n}=0}}{\rm sgn}(\sigma)\prod_{i=1}^{n}h(a_{i},\mathcal{S}^{k_{i}}b_{\sigma(i)}).
\end{align*}
\end{theorem}

\begin{figure}[tb]
\includegraphics[scale=1.2]{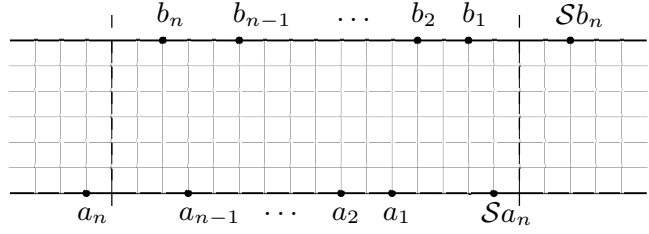}
\caption{The lattice strip $G$ with vertex set $V=\mathbb{Z}\times\{0,1,...,N\}$.}
\label{peru}
\end{figure}

\begin{remark}
Unlike Fomin's identity, the extra condition $P_{1}\cap LE(\mathcal{S}P_{n})=\emptyset$ in (\ref{Fomin2}) forces us to consider families of paths where the $n$ end vertices are permutations {\it and} translations of the originals $(b_{1},...,b_{n})$, see the proof below. In particular, the end vertices should vary among $n$-tuples ${\bf \mathcal{S}^{k}b_{\sigma}}=(\mathcal{S}^{k_{1}}b_{\sigma(1)},...,\mathcal{S}^{k_{n}}b_{\sigma(n)})$, with $\sigma\in S_{n}$ and $k_{1}+...+k_{n}=0$, $k_{i}\in\mathbb{Z}$. This can be thought of as the action of the (infinite) affine symmetric group $\tilde{A}_{n}$ on the vertices $(b_{1},...,b_{n})$. 
\end{remark}

\begin{remark}
In the acyclic case, the above theorem agrees with the Gessel-Zeilberger formula for counting paths in alcoves \cite{Gessel2}.
\end{remark}

\begin{proof}[Proof of Theorem \ref{Fominaff}]
We will follow the strategy of proof of Fomin's identity (Theorem 6.1 in \cite{Fomin}), that is, we will give a sign-reversing involution on the set of summands on the right hand side of (\ref{Fomin2}) which violate the condition
\begin{align}\label{rusia}
P_{j}\cap & LE(P_{i})=\emptyset,\quad\text{for all}\quad 1\leq i<j\leq n,\quad\text{and}\\\notag
&P_{j}\cap LE(\mathcal{S}P_{n})=\emptyset,\quad\text{for all}\quad 1\leq j\leq n.
\end{align}
As a consequence, the sum of all of the latter terms will vanish and, the sum of the remaining terms, the ones which satisfy (\ref{rusia}), will be simplified to the desired expression in the left-hand side of (\ref{Fomin2}). 

The sign-reversing involution is as follows. For $n\geq2$, let $\sigma\in S_{n}$ 
 and $k_{i}$, $1\leq i\leq n$, integers such that $k_{1}+k_{2}+...+k_{n}=0$. Consider a family of paths ${\bf P}\in\mathcal{H}({\bf a},{\bf \mathcal{S}^{k}b_{\sigma}})$ which violates the condition (\ref{rusia}). We will construct a new family of paths $\tilde{\bf P}\in\mathcal{H}({\bf a},{\bf \mathcal{S}^{\tilde k}b_{\tilde\sigma}})$, with $\tilde\sigma\in S_{n}$ and $\tilde k_{1}+\tilde k_{2}+...+\tilde k_{n}=0$, that also violates the condition (\ref{rusia}), and satisfes $\omega(\tilde{\bf P})=\omega({\bf P})$ and sgn$(\tilde\sigma)=-$sgn$(\sigma)$. The construction of the new family $\tilde{\bf P}$ is essentially an application of the Fomin's loop-erased switching (Section \ref{rplerw}) over the paths
\begin{align*}
P_{n},P_{n-1},...,P_{1},\mathcal{S}P_{n}.
\end{align*} 
This construction will also ensure that the correspondence
${\bf P}\mapsto\tilde{\bf P}$ is one-to-one, as desired.  

To make the notation simpler, let us denote $a_{0}:=\mathcal{S}a_{n}$ and the corresponding path starting at $\mathcal{S}a_{n}$ by $P_{0}:=\mathcal{S}P_{n}$. Choose indexes $i'$ and $j'$ as follows. Since the family ${\bf P}\in\mathcal{H}({\bf a},{\bf \mathcal{S}^{k}b_{\sigma}})$ violates (\ref{rusia}), the set of indexes $0\leq i<j\leq n$ such that $P_{j}\cap LE(P_{i})\not=\emptyset$ is not empty. Therefore, we can choose $0\leq i^{'}< n$ the minimum among those indexes and consider the path LE$(P_{i'})$. Along the latter path, choose a vertex $v'$ and index $j'$ as follows:
\begin{itemize}
	\item Along the vertex sequence of the path LE$(P_{i'})$, choose $v'$ as the `closest' (that is, with minimal index) intersection vertex to the starting vertex $a_{i'}$.
	\item Now consider the set of indexes $\{j:1\leq i'<j\leq n\}$ such that $P_{j}$ intersects LE$(P_{i'})$ at $v'$ (in other words, $v'\in P_{j}\cap\text{LE}(P_{i'})$), and let $j'$ the minimum of this set.
\end{itemize}

We have two different scenarios, depending on weather $P_{i'}$ is the path $\mathcal{S}P_{n}$ or not. If $i'\not=0$ (and $P_{i'}$ is not the path $\mathcal{S}P_{n}$), we perform the usual loop-erased switching (Section \ref{rplerw}) over the paths $P_{i'}$ and $P_{j'}$ at the vertex $v'$, that is, we define new paths
\begin{align*}
\tilde P_{i'}: a_{i'}&\xrightarrow{P_{i'}'}v'\xrightarrow{P_{j'}''}\mathcal{S}^{k_{j'}}b_{\sigma(j')}\\
\tilde P_{j'}: a_{j'}&\xrightarrow{P_{j'}'}v'\xrightarrow{P_{i'}''}\mathcal{S}^{k_{i'}}b_{\sigma(i')}.
\end{align*}
For the remaining paths, $i\notin\{i',j'\}$, we define $\tilde P_{i}:=P_{i}$. The original family ${\bf P}\in\mathcal{H}({\bf a},{\bf \mathcal{S}^{k}b_{\sigma}})$ is then mapped to a new family of paths $\tilde{\bf P}\in\mathcal{H}({\bf a},{\bf \mathcal{S}^{\tilde k}b_{\tilde\sigma}})$, where $\tilde{\bf k}=(k_{1},...,k_{j'},...,k_{i'},...,k_{n})$ and $\tilde\sigma=\sigma\circ(i',j')\in S_{n}$ are the vector ${\bf k}$ and permutation $\sigma$, with the entries $i'$ and $j'$ interchanged. Note that the sum of the entries of $\tilde{\bf k}$ is zero, as desired, and ${\rm sgn}(\tilde\sigma)=-{\rm sgn}(\sigma)$. Moreover, the family $\tilde{\bf P}$ also violates the condition (\ref{rusia}) since the paths $\tilde P_{i'}$ and $\tilde P_{j'}$ share the vertex $v'$. Note that the weights are also preserved: $\omega(\tilde{\bf P})=\omega({\bf P})$.

\begin{figure}[tb]
\includegraphics[scale=1.2]{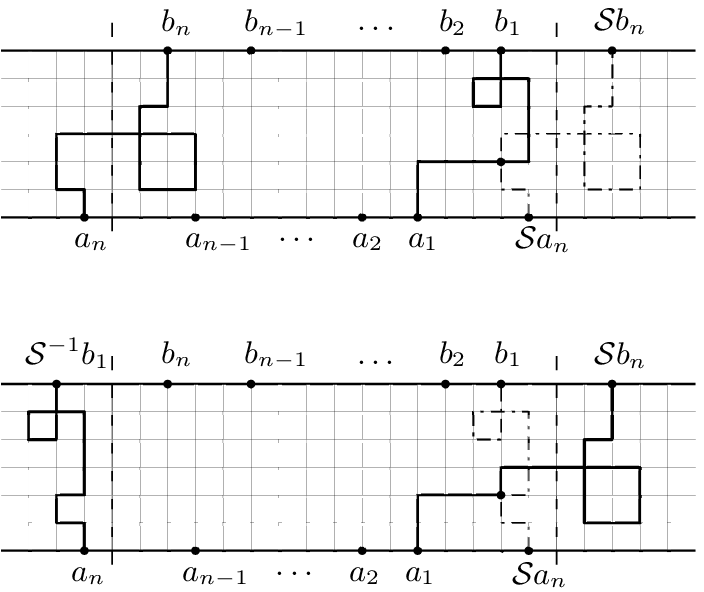}
\caption{Loop-erased switching over the paths $P_1$ and $\mathcal{S}P_n$.}
\label{croacia}
\end{figure}

In the second case, when $i'=0$, a more careful selection of paths is needed: we perform the loop-erased switching over the paths $\mathcal{S}P_{n}$ and $P_{j'}$:
\begin{align*}
\mathcal{S}P_{n}&: \mathcal{S}a_{n}\xrightarrow{(\mathcal{S}P_{n})'}v'\xrightarrow{(\mathcal{S}P_{n})''}\mathcal{S}\mathcal{S}^{k_{n}}b_{\sigma(n)}\\
P_{j'}&: a_{j'}\xrightarrow{P_{j'}'}v'\xrightarrow{P_{j'}''}\mathcal{S}^{k_{j'}}b_{\sigma(j')},
\end{align*}
and create the two new paths
\begin{align*}
\tilde P_{n}&: a_{n}\xrightarrow{P_{n}'}\mathcal{S}^{-1}v'\xrightarrow{\mathcal{S}^{-1}P_{j'}''}\mathcal{S}^{-1}\mathcal{S}^{k_{j'}}b_{\sigma(j')}\\
\tilde P_{j'}&: a_{j'}\xrightarrow{P_{j'}'}v'\xrightarrow{(\mathcal{S}P_{n})''}\mathcal{S}\mathcal{S}^{k_{n}}b_{\sigma(n)},
\end{align*}
(see Figure \ref{croacia}). The rest of the paths remain invariant, $\tilde P_{i}:=P_{i}$ for $i\notin\{n,j'\}$. Thus, the new family $\tilde{\bf P}=(\tilde{P}_1,...,\tilde{P}_n)$ satisfies $\tilde{\bf P}\in\mathcal{H}({\bf a},{\bf \mathcal{S}^{\tilde k}b_{\tilde\sigma}})$, with $\tilde{\bf k}=(k_{1},...,k_{n}+1,...,k_{n-1},k_{j'}-1)$ and $\tilde\sigma=\sigma\circ(j',n)\in S_{n}$. Note again that the sum of the entries of $\tilde{\bf k}$ is zero and ${\rm sgn}(\tilde\sigma)=-{\rm sgn}(\sigma)$. Moreover, since the weight $\omega$ is invariant under horizontal translations, we have $\omega(\tilde{\bf P})=\omega({\bf P})$. We only need to show that the family $\tilde{\bf P}$ violates the condition (\ref{rusia}) as well, but this is clearly the case since the paths $\tilde P_{j'}$ and $LE(\mathcal{S}\tilde P_{n})$ intersect at the vertex $v'$, that is
\begin{align*}
\tilde P_{j'}\cap LE(\mathcal{S}\tilde P_{n})\not=\emptyset.
\end{align*}

Therefore, in both cases, applying the loop-erased switching to the family $\tilde{\bf P}$, we recover the original family ${\bf P}$, so the corresponding map from ${\bf P}$ to $\tilde{\bf P}$ is an involution on the set of paths that violate (\ref{rusia}). Moreover, since 
${\rm sgn}(\sigma)\omega({\tilde{\bf P}})=-{\rm sgn}(\tilde\sigma)\omega(\tilde{\bf P})$, the sum of all these terms vanishes on the right hand side of (\ref{Fomin2}), and therefore the total sum is 
\begin{align}\label{elaborada}
\sum_{\sigma\in S_{n}}\sum_{\substack{k_{i}\in\mathbb{Z}\\k_{1}+k_{2}+...+k_{n}=0}}{\rm sgn}(\sigma)\sum_{\substack{{\bf P}\in\mathcal{H}^{+}({\bf a},\mathcal{S}^{\bf k}{\bf b}_{\sigma}) \\ {\bf P}\,\,\text{satisfies}\,\,(\ref{rusia})}}\omega({\bf P}).
\end{align}
Finally, in the expression above, if a family ${\bf P}\in\mathcal{H}^{+}({\bf a},\mathcal{S}^{\bf k}{\bf b}_{\sigma})$ satisfies (\ref{rusia}), the loop-erased parts $LE(P_{j})$, $1\leq j\leq n$, are pairwise disjoint and then $\sigma$ must be the identity permutation and $k_{1}=k_{2}=...=k_{n}=0$, as required. In this case, the condition (\ref{rusia}) on paths can be simplified to the one in the left hand side of (\ref{Fomin2}).

\end{proof}

\subsection{Projections onto the cylinder}
\label{projections}
As described in the introduction of Section \ref{avfi}, a useful application of Theorem \ref{Fominaff} is when we consider the projection of the lattice strip $G$ onto the {\it cylindrical lattice}, modulo a translation $\mathcal{S}$ (see Figure \ref{cylind}). Intuitively, a family of $n$ (loop-erased) paths can wind around the cylinder several times (equivalently, translations of the end vertex by $\mathcal{S}^{m}=\mathcal{S}\circ\mathcal{S}^{m-1}$, $m\in\mathbb{Z}$, in the strip) before reaching its destination. Moreover, there are exactly $n$ different ways in which the $n$ paths can reach their destination without intersecting, given by the $n$ `cyclic permutations' of the end vertices.

\begin{figure}[tb]
\includegraphics[scale=0.6]{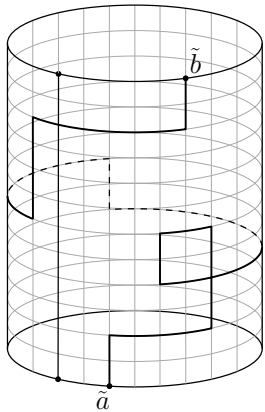}
\caption{A path in the cylindrical lattice $\tilde{G}$ with winding number $k=1$.}
\label{cylind}
\end{figure}

Corollary \ref{grand} and Proposition \ref{Fominaff3} make the above considerations precise. These considerations give a more tractable form of Theorem \ref{Fominaff}, first as a sum of $n$ determinants in Corollary \ref{grand} and then as a single determinant in Proposition \ref{Fominaff3}.

\begin{cor}\label{grand}
In the context of Theorem \ref{Fominaff}, by summing up in (\ref{Fomin2}) over all the weights of all families of paths ${\bf P}=(P_{1},...,P_{n})$ starting at ${\bf a}=(a_{1},...,a_{n})$, and ending at all possible translations of ${\bf b}=(b_{1},...,b_{n})$ by $\mathcal{S}^{m}$, $m\in\mathbb{Z}$, we obtain
\begin{align}\label{cylinder}
\sum_{\substack{{\bf P}\in\bigcup_{m\in\mathbb{Z}}\mathcal{H}^{+}({\bf a},{\bf \mathcal{S}^{m}}{\bf b}) \\ P_{j}\cap\text{LE}(P_{j-1})=\emptyset,\quad 1< j\leq n \\ P_{1}\cap\text{LE}(\mathcal{S}P_{n})=\emptyset}}\omega({\bf P})=\sum_{\sigma\in S_{n}}\sum_{\substack{k_{1}+k_{2}+...+k_{n}=0\\\rm{mod}\,\, n}}{\rm sgn}(\sigma)\sum_{{\bf P}\in\mathcal{H}^{+}({\bf a},{\bf \mathcal{S}^{k}b_{\sigma}})}\omega({\bf P}),
\end{align}
where ${\bf \mathcal{S}^{m}b}=(\mathcal{S}^{m}b_{1},...,\mathcal{S}^{m}b_{n})$. Moreover, if $\eta=e^{i\frac{2\pi}{n}}$ is a complex root of unity, then the right hand side above can be expressed as the sum of $n$ determinants:
\begin{align}\label{simpl}
\frac{1}{n}\sum_{u=0}^{n-1}\det(\sum_{k\in\mathbb{Z}}\eta^{uk}h(a_{i},\mathcal{S}^{k}b_j))_{i,j=1}^{n},
\end{align}
where $h(a,b)=\sum_{P\in\mathcal{H}^{+}(a,b)}\omega(P)$, $a\in V, b\in\partial\Gamma$.
\end{cor}

\begin{proof}
Using the identity (\ref{Fomin2}) and summing up over all the weights as indicated in the statement of the corollary, the left hand side of (\ref{cylinder}) takes the form
\begin{align*}
\sum_{m\in\mathbb{Z}}\sum_{\sigma\in S_{n}}\sum_{\substack{k_{i}\in\mathbb{Z}\\k_{1}+k_{2}+...+k_{n}=0}}{\rm sgn}(\sigma)\sum_{{\bf P}\in\mathcal{H}^{+}({\bf a},{\bf \mathcal{S}^{m+k}b_{\sigma}})}\omega({\bf P}),
\end{align*}
which, in turn, can be easily simplified to the desired expression on the right-hand side. For the second part, note that, if $\eta=e^{i\frac{2\pi}{n}}$ is a complex root of unity, then we can eliminate the condition $\sum_{i=1}^{n} k_{i}=0$ mod $n$ by using the identity
\begin{align*}
\frac{1}{n}\sum_{u=0}^{n-1}\eta^{u\sum_{i=1}^{n} k_{i}}=\left\{
	\begin{array}{ll}
		1  & \mbox{if } \sum_{i=1}^{n}k_{i}=0,\,\,\text{mod}\,n \\
		0 & \mbox{otherwise}.
	\end{array}
\right.
\end{align*}
Then, the right-hand side of (\ref{cylinder}) can be written as
\begin{align*}
\frac{1}{n}\sum_{u=0}^{n-1}\sum_{\sigma\in S_{n}}{\rm sgn}(\sigma)\sum_{k_{1},...,k_{n}\in\mathbb{Z}}\eta^{u\sum_{i=1}^{n} k_{i}}\sum_{{\bf P}\in\mathcal{H}^{+}({\bf a},{\bf \mathcal{S}^{k}b_{\sigma}})}\omega({\bf P}),
\end{align*}
and the latter as
\begin{align*}
\frac{1}{n}\sum_{u=0}^{n-1}\sum_{\sigma\in S_{n}}{\rm sgn}(\sigma)\prod_{i=1}^{n}\sum_{k\in\mathbb{Z}}\eta^{uk}\sum_{P\in\mathcal{H}^{+}(a_{i},\mathcal{S}^{k}b_{\sigma(i)})}\omega(P).
\end{align*}
The above expression is (\ref{simpl}).

\end{proof}

\begin{proposition}\label{Fominaff3}
Denote by $[\ell]\in S_n$ the cyclic permutation shifted by $\ell=0,1,...,n-1$ :
$$[\ell](k)=k-\ell,\,\,\mod n,\,\,\,{\rm in}\,\,\,\{1,...,n\}.$$
Let ${\bf a}=(a_{1},...,a_{n})$ and ${\bf b}=(b_1,b_2,...,b_n)$ be the vectors of vertices of Theorem \ref{Fominaff}. For each $[\ell]\in S_{n}$, $\ell=0,...,n-1$, define the $n$-tuple:
\begin{align}\label{cyc}
{\bf k_{\ell}}=(\underbrace{1,...,1,}_{\ell\,\,\rm{times}}\underbrace{0,...,0}_{n-\ell\,\,\rm{times}}).
\end{align}
We have the following 
\begin{align}\label{weighted2}
\sum_{\substack{[\ell]\in S_{n}\\ \ell=0,...,n-1}}\sum_{\substack{{\bf P}\in\bigcup_{m\in\mathbb{Z}}\mathcal{H}^{+}({\bf a},{\bf \mathcal{S}^{m+k_{\ell}}}{\bf b}_{[\ell]}) \\ P_{j}\cap\text{LE}(P_{j-1})=\emptyset,\quad 1< j\leq n \\ P_{1}\cap\text{LE}(\mathcal{S}P_{n})=\emptyset}}\omega({\bf P})=\det\left(\sum_{k\in\mathbb{Z}}\zeta^{k}h(a_{i},\mathcal{S}^{k}b_{j})\right)_{i,j=1}^{n},
\end{align}
where  
$$
\zeta= \left\{
        \begin{array}{ll}
            1 & \quad \text{if $n$ is odd} \\
            -1 & \quad \text{if $n$ is even},
        \end{array}
    \right.
$$
and $h(a,b)=\sum_{P\in\mathcal{H}^{+}(a,b)}\omega(P)$. In particular, if the weight function $\omega$ is non-negative, then the above determinant is non-negative.
\end{proposition}

\begin{proof}
Let $\mathcal{G}({\bf a},{\bf b})$ denote the left-hand side of (\ref{cylinder}). Using Corollary \ref{grand}, a simple calculation shows that for each $\ell=0,...,n-1$:
\begin{align*}
\mathcal{G}({\bf a},{\bf \mathcal{S}^{k_{\ell}}b}_{[\ell]})=\frac{1}{n}\sum_{u=0}^{n-1}\eta^{-\ell u}{\rm sgn}([\ell])\det(\sum_{k\in\mathbb{Z}}\eta^{uk}h(a_{i},\mathcal{S}^{k}b_j)).
\end{align*}
Therefore, the left hand side of (\ref{weighted2}) can be expressed as
\begin{align*}
\sum_{\ell=0}^{n-1}\mathcal{G}({\bf a},{\bf \mathcal{S}^{k_{\ell}}b}_{[\ell]})=\sum_{u=0}^{n-1}\left(\frac{1}{n}\sum_{\ell=0}^{n-1}\eta^{-\ell u}{\rm sgn}([\ell])\right)\det(\sum_{k\in\mathbb{Z}}\eta^{uk}h(a_{i},\mathcal{S}^{k}b_j)),
\end{align*}
which is a sum of $n$ determinants.

Case 1. If $n$ is odd, sgn$([\ell])=1$ for all $\ell=0,...,n-1$ and then
\begin{align*}
\frac{1}{n}\sum_{\ell=0}^{n-1}\eta^{-\ell u}=\left\{
	\begin{array}{ll}
		1  & \mbox{if } -u=0,\,\,\text{mod}\,n \\
		0 & \mbox{otherwise},
	\end{array}
\right.
\end{align*}
therefore, the only remaining determinant is the one corresponding to $u=0$, and so $\zeta=\eta^{u}=1$.

Case 2. If $n$ is even, sgn$([\ell])=(-1)^{\ell}$ for all $\ell=0,...,n-1$ and
\begin{align*}
\frac{1}{n}\sum_{\ell=0}^{n-1}(-1)^{\ell}\eta^{-\ell u}=\frac{1}{n}\sum_{\ell=0}^{n-1}\eta^{\ell(\frac{n}{2}-u)}.
\end{align*}
The above sum is $1$ if and only if $\frac{n}{2}-u=0$ mod $n$, and zero otherwise. The only remaining determinant is then $u=\frac{n}{2}$, and therefore $\zeta=\eta^{u}=-1$, which concludes the proof.

\end{proof}

\begin{remark}
\label{Fominaff355}
Assume that the vertex set $V=\mathbb{Z}\times\{0,1,...,N\}$ is the state space of a time-homogeneous Markov chain $X$ and the possible transitions between states are determined by the lattice strip $G$ introduced at the beginning of the section. In other words, the transition probabilities $p(u,v)$ are positive if and only if there is an edge $u\stackrel{e}{\to}v$, in which case $\omega(e)=p(u,v)$. Assume that the transition probabilities are space-invariant with respect to a fixed horizontal translation $\mathcal{S}:G\to G$.
Then, the assertion of Proposition \ref{Fominaff3} has the following probabilistic interpretation: if vertices $a_{n},...,a_{1},b_{1},...,b_{n}$ are ordered counterclockwise along the boundary (as in Figure \ref{croacia}), then the $n\times n$ determinant
\begin{align*}
\det\left(\sum_{k\in\mathbb{Z}}\zeta^{k}h(a_{i},\mathcal{S}^{k}b_{j})\right)_{i,j=1}^{n},
\end{align*}
of hitting probabilities $h(a_{i},\mathcal{S}^{k}b_{j})$, where $\mathcal{S}^{k}=\mathcal{S}\circ\mathcal{S}^{k-1}$, $k\in\mathbb{Z}$, and 
$$
\zeta= \left\{
        \begin{array}{ll}
            1 & \quad \text{if $n$ is odd} \\
            -1 & \quad \text{if $n$ is even},
        \end{array}
    \right.
$$
is equal to the probability that $n$ independent trajectories of the Markov chain $X_{1},...,X_{n}$, starting at $a_{1},...,a_{n}$, respectively, will first hit the upper boundary $\partial\Gamma=\mathbb{Z}\times\{N\}$ at any of the $n$ cyclic permutations of the vertices $b_{1},...,b_{n}$, shifted also by all possible horizontal translations by $\mathcal{S}^{k}$, $k\in\mathbb{Z}$, and furthermore the trajectories are constrain to satisfy
\begin{align*}
X_{j}\cap LE(X_{j-1})=\emptyset,\quad 1< j\leq n,\quad\text{and}\quad X_{1}\cap LE(\mathcal{S}X_{n})=\emptyset.
\end{align*} 

\end{remark}

\subsection{A remark on loop-erased paths in a cylinder}
\label{lewcyl}

In this section, we consider families of paths defined in the directed cylindrical lattice $\tilde G$ of Figure \ref{cylind} and review some properties regarding their loop-erasures. As in the previous sections,  the cylindrical lattice need not be acyclic (loops are allowed) and, if the number of paths is odd, we can obtain a variant of Proposition \ref{Fominaff3}  by applying Fomin's identity directly (see Proposition \ref{Fominaff2} below). However, there is a slight difference between these two approaches, since the loop-erasure of a path in $\tilde G$ may differ from the projection of the loop-erasure of the corresponding path in the lattice strip $G$ (see the remark just after Proposition \ref{Fominaff2}).

Define the (directed) {\it cylindrical lattice} (or, just {\it cylinder}) $\tilde{G}=(\tilde{V},\tilde{E})$ as the directed graph with vertex set $\tilde{V}=\mathbb{Z}_{M}\times\{0,1,...,N\}$ and connected by edges in both positive and negative directions. Here, we consider the canonical representation of $\mathbb{Z}_{M}$ as $\mathbb{Z}/M\mathbb{Z}=\{[0],...,[M-1]\}$. Let's distinguish the set of (boundary) vertices $\partial\tilde{G}=\mathbb{Z}_{M}\times\{N\}$. If we consider the lattice strip $G=(V, E, \omega)$ of Theorem \ref{Fominaff} and the notation thereof, there is a natural correspondence between paths in the cylinder $\tilde{G}$ and paths in the strip $G$. In particular, every path $\tilde{P}$ in $\tilde{G}$ starting at $\tilde{a}=([i],0)$ and ending at $\tilde{b}=([j],N)$, for $i,j\in\{0,...,M-1\}$, and with all internal vertices lying in $\tilde{G}\setminus\partial\tilde{G}$, can be seen as the image of any path $P^{\ell}\in G$ of the form
\begin{align*}
P^{\ell}\in\mathcal{H}^{+}(\mathcal{S}^{\ell}a,\mathcal{S}^{\ell+k}b),\quad\ell\in\mathbb{Z},
\end{align*}
with $a=(i,0)\in V$, $b=(j,N)\in V$, and a unique $k\in\mathbb{Z}$. The integer $k$ is usually called the {\it winding number} of the path $\tilde{P}$ (see Figure \ref{cylind}). Since the weight function $\omega$ defined on the strip $G$ is invariant under the translation $\mathcal{S}$ by $(M,0)$, the graph $\tilde{G}$ inherits canonically a weight function $\tilde{\omega}$ on $\tilde{E}$ and the path $\tilde{P}$ inherits the weight $\tilde{\omega}(\tilde{P}):=\omega(P)$, whenever $P\in\mathcal{H}^{+}(a,\mathcal{S}^{k}b)$ is the projection of $\tilde{P}$.

Let $\mathcal{C}^{+}(\tilde{a},\tilde{b})$ be the set of all paths in the cylinder $\tilde{G}$ of positive length, starting at $\tilde{a}\in\tilde{G}$ and ending at $\tilde{b}\in\partial\tilde{G}$, with all internal vertices in $\tilde{G}\setminus\partial\tilde{G}$. Similarly, define $\mathcal{C}^{+}(\tilde{\bf a},\tilde{\bf b})$ for families of paths $\tilde{P}_{1},...,\tilde{P}_{n}$, starting at $\tilde{\bf a}=(\tilde{a}_1,...,\tilde{a}_n)$ and ending at $\tilde{\bf b}=(\tilde{b}_1,...,\tilde{b}_n)$. We have the following:

\begin{proposition}\label{Fominaff2}
Consider integers $0\leq i_{n}<i_{n-1}<...<i_{1}<M$ and $0\leq j_{n}<j_{n-1}<...<j_{1}<M$. Define two $n$-tuples of vertices in the cylinder $\tilde{G}$ as
\begin{align*}
\tilde{a}_{k}:=([i_{k}],0),\quad\quad \tilde{b}_{k}:=([j_{k}],N),\quad 1\leq k\leq n.
\end{align*} 
If $n$ is odd and we consider the $n$ cyclic permutations defined in Proposition \ref{Fominaff3}, we obtain

\begin{align}\label{weighted}
\sum_{\sigma\,\,\rm{cyclic}}\sum_{\substack{\tilde{\bf P}\in\mathcal{C}^{+}(\tilde{\bf a},\tilde{\bf b}_{\sigma})\\ \tilde{P}_{j}\cap LE(\tilde{P}_{i})=\emptyset,\,\,i<j}}\tilde{\omega}(\tilde{\bf P})=\det\left(\sum_{k\in\mathbb{Z}}h(a_{i},\mathcal{S}^{k}b_{j})\right)_{i,j=1}^{n},
\end{align}
where $h(a,b)=\sum_{P\in\mathcal{H}^{+}(a,b)}\omega(P)$ and $\tilde{\bf b}_{\sigma}=(\tilde{b}_{\sigma(1)},...,\tilde{b}_{\sigma(n)})$.

\end{proposition}

\begin{figure}[tb]
\includegraphics[scale=0.51]{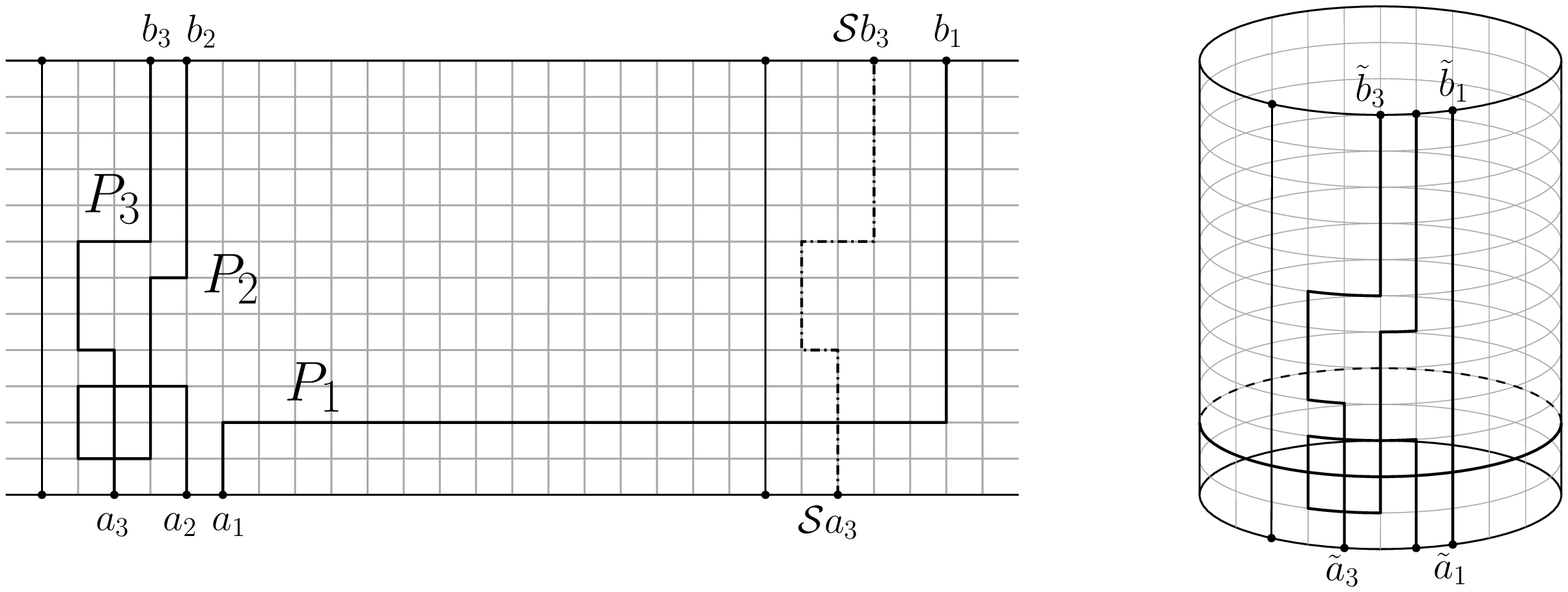}
\caption{Three paths, and their projections onto the cylinder.}
\label{cylind2}
\end{figure}

\begin{proof}
Note that the right hand side of (\ref{weighted}) can be written as the determinant
\begin{align*}
\det(\sum_{\tilde{P}\in\mathcal{C}^{+}(\tilde{a}_{i},\tilde{b}_{j})}\tilde{\omega}(\tilde{P}))_{i,j=1}^{n},
\end{align*}
and, since sgn$(\sigma)$=1 for all $\sigma$ cyclic if $n$ is odd, the equality (\ref{weighted}) is a direct application of Fomin's identity (Theorem \ref{Fomingen}), according to the weight function $\tilde{\omega}$ on $\tilde{G}$. 

\end{proof}

\begin{remark}\label{abc}
The right hand side of (\ref{weighted}) agrees with the right hand side of identity (\ref{weighted2}), for $n$ odd. This implies that the left hand sides of (\ref{weighted}) and (\ref{weighted2}) are equal, which is not immediately obvious from the definitions. For example, we can consider the paths $P_{1}$, $P_{2}$ and $P_{3}$ in the lattice strip of Figure \ref{cylind2}. There, we have that the corresponding projections onto the cylinder satisfy $\tilde{P}_{j}\cap LE(\tilde{P}_{i})=\emptyset$, for $i<j$, and, in particular $\tilde{P}_{3}\cap LE(\tilde{P}_{1})=\emptyset$. However, $P_{1}\cap LE(\mathcal{S}P_{3})\not=\emptyset$, and then $(P_{1},P_{2},P_{3})$ is not considered in the left hand side of (\ref{weighted2}). It would be interesting to have a direct combinatorial proof of this identity.
\end{remark}

\begin{remark}
In the acyclic case, one can obtain determinant formulas for an {\it even} number of non-intersecting walks
on a cylindrical lattice by introducing modified weights which keep track of windings~\cite{Fulmek,Liechty}.
However, in the general case, we do not see how to adopt this approach and the only way we know
how to study the case of an even number of particles is via the affine version of Fomin's identity
introduced in Theorem \ref{Fominaff}.
\end{remark}

\subsection{More general lattices, Gessel-Zeilberger formula}
\label{generallatt}

The  results of this section were formulated for the square lattice but are equally valid for more general periodic planar graphs, for example, the hexagonal lattice shown in Figure \ref{diamond}.

In the acyclic case, Theorem \ref{Fominaff} agrees with the Gessel-Zeilberger formula \cite{Gessel2} (see also \cite{Gessel}). We note that in this context, the identity (\ref{weighted2})
gives a direct connection between the Gessel-Zeilberger formula, for counting paths in alcoves,
and the Karlin-McGregor formula \cite{Karlin} (Lindstr\"{o}m-Gessel-Viennot lemma \cite{Gessel-Viennot}) for counting non-intersecting paths on a cylinder; this answers positively a question of Fulmek \cite{Fulmek}, where the problem of finding 
such a direct connection was posed as an open question.  Moreover, in the continuous case, 
it also shows that the Karlin-McGregor (for $n$ odd) and Liechty-Wang (for $n$ even)
formulas \cite{Karlin,Liechty} for the transition probability density of $n$ (indistinguishable)
non-intersecting Brownian motions on the circle can be obtained directly 
from the (labelled) model of Hobson-Werner \cite{Werner}, which is a continuous version of
the Gessel-Zeilberger formula in the case of the affine symmetric group $\tilde{A}_{n}$ (we review this in Section \ref{bmuc} below).

\begin{figure}[tb]
\includegraphics[scale=0.4]{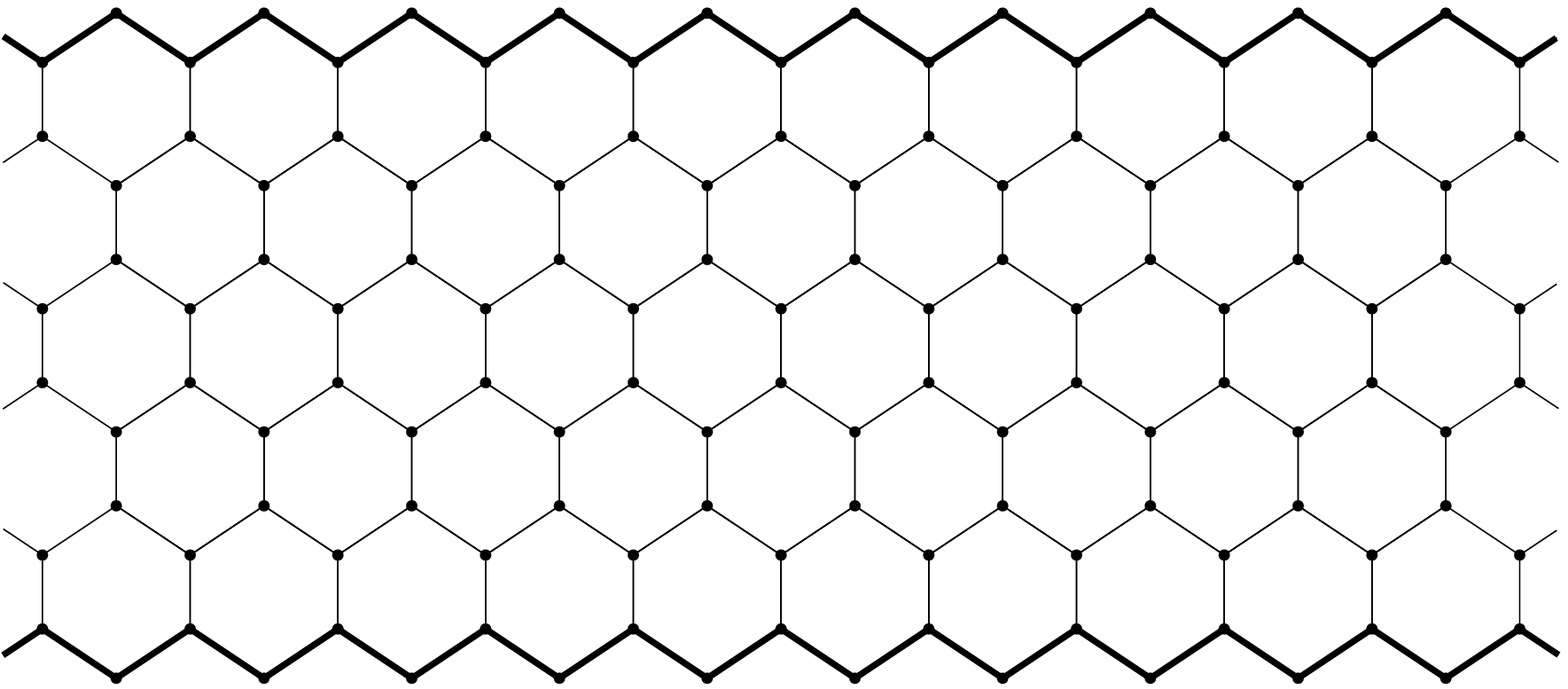}
\caption{Hexagonal lattice}
\label{diamond}
\end{figure}

\section{Connections to random matrix theory}
\label{rmt}

As explained in Section \ref{sclim} of the introduction, there is a natural way to consider diffusion scaling limits of both Fomin's identity (Corollary \ref{Fomin}) and its affine version (Proposition \ref{Fominaff3}). Regarding Fomin's identity, this idea is originally discussed in \cite{Fomin}, where some examples for two-dimensional Brownian
motion are described in detail. For our purposes, the connection with random matrix theory emerges from the following considerations: assume that $\Omega$ is a suitable complex (connected) domain with {\it smooth} boundary and $h(z_{0},y)$ is the (hitting) density of the harmonic measure
$$\mu_{z_{0},\Omega}(A)=\mathbb{P}^{z_{0}}(B_{T}\in A),\quad A\subset\partial\Omega,$$
with respect to one-dimensional Lebesgue measure (lenght), where $B$ under $\mathbb{P}^{z_{0}}$ denotes a two-dimensional Brownian motion starting at $z_{0}\in\Omega$, and $T=\inf\{t>0:B_{t}\notin\Omega\}$ is the first exit time of $\Omega$ (see Section \ref{abrociobm}). Therefore, for $m\in\mathbb{R}$ and appropriately chosen (parametrized) positions $x_{1},...,x_{n}$ and $y_{1},...,y_{n}$ along the boundary $\partial\Omega$, the {\it determinants of hitting densities}
\begin{align}\label{Poissonkernel4}
H(x,y)=\det\left(h(x_{i},y_{j})\right)_{i,j=1}^{n}dy_{1}\cdots dy_{n},
\end{align}
and
\begin{align}\label{Poissonkernel24}
H(x,y)=\det\left(\sum_{k\in\mathbb{Z}}\zeta^{k}h(x_{i},y_{j}+mk)\right)_{i,j=1}^{n}dy_{1}\cdots dy_{n},
\end{align}
where  
$$
\zeta= \left\{
        \begin{array}{ll}
            1 & \quad \text{if $n$ is odd} \\
            -1 & \quad \text{if $n$ is even},
        \end{array}
    \right.
$$
can be interpreted, informally, as the probability that $n$ independent `Brownian motions' $B_{i}$, $i=1,...,n$, starting at positions $x_{1},x_{2},...,x_{n}$, respectively, will first hit an absorbing boundary $\partial\Gamma\subset\partial\Omega$ at (parametrized) positions in the intervals $(y_{i}+dy_{i})$, $i=1,...,n$, and whose trajectories are constrained to satisfy
$$
B_{j}\cap LE(B_{i})=\emptyset,\quad\text{for all}\,\,1\leq i<j\leq n,
$$ 
in (\ref{Poissonkernel4}), or
\begin{align*}
B_{j}\cap LE(B_{j-1})=\emptyset,\quad 1< j\leq n,\quad\text{and}\quad B_{1}\cap LE(m+B_{n})=\emptyset,
\end{align*}
in the affine case (\ref{Poissonkernel24}). We remark again that, in the affine case, we assume $\Omega$ to be invariant under a fixed (horizontal) translation by $m\in\mathbb{R}$, and therefore $m+B_{n}$ is the horizontal translation by $m$ of the Brownian path $B_{n}$.

Our main interest is the determination of the behaviour of the
$n$ hitting points $y_{1},...,y_{n}$ along the boundary, when the starting points $x_{1},...,x_{n}$ merge into a single common point in $\partial\Omega$. In other words, for determinants of the form (\ref{Poissonkernel4}), this section considers certain limits
\begin{align*}
\lim_{\substack{(x_{1},...,\,x_{n})\in C\\x_{i}\to x\in\partial\Omega}}\tilde{H}(x,y),\quad (y_{1},...,y_{n})\in C,
\end{align*}
where $\tilde{H}(x,y)$ is an appropriate normalisation of $H(x,y)$ and the positions $x_{1},...,x_{n}$, $y_{1},...,y_{n}$ are determined by chambers $C$ of $\mathbb{R}^{n}$. Determinants of the affine form (\ref{Poissonkernel24}) are considered in Section \ref{circular}.

In Sections \ref{bmpq}, \ref{bmstrip} and \ref{unitdisk}, we revisit the examples considered in \cite{Fomin} (see Figure \ref{dom}). 
We will see that the consideration of the above limits reveals some natural connections to random matrices, particularly Cauchy type ensembles \cite{ForresterWitte}. An example of this connection was first observed by Sato and Katori \cite{Katori}, in the context of {\it excursion Poisson kernel determinants}, and we discuss this in Section \ref{nnoepkd}. Section \ref{circular} considers the affine (circular) case and shows that it is also
related in a natural way to circular ensembles of random matrix theory.

As a warm-up, in Section \ref{an} we recall a well-known connection between non-intersecting {\it one-dimensional}
Brownian motions and the Gaussian Orthogonal Ensemble (GOE) of random matrix theory.

\subsection{A brief review on conformal invariance of Brownian motion}
\label{abrociobm}
The Riemann mapping theorem asserts that any two proper simply connected domains of $\mathbb{C}$ can be conformally mapped into each other. More precisely, if $\Omega\subset\mathbb{C}$ and $\Omega'\subset\mathbb{C}$ are two proper simply connected domains with $z_{0}\in\Omega$ and $z_{0}'\in \Omega'$, then there exists a unique  conformal (analytic with non-vanishing derivative) map $f:\Omega\to \Omega'$ such that $f(z_{0})=z_{0}'$ and $f'(z_{0})>0$. In addition, it is well known that the two-dimensional Brownian motion is invariant under conformal transformations \cite{Durrett}:

\begin{proposition}
If $B$ is a two-dimensional Brownian motion starting at $z_{0}\in \Omega$ and $T=\inf\{t>0: B_{t}\notin \Omega\}$ is the exit time of the domain $\Omega$, then there exists a (random) time change $\sigma:[0,T']\to[0,T]$ such that the process
\begin{align*}
(f(B_{\sigma(t)}),0\leq t<T')
\end{align*}
is again a two-dimensional Brownian motion, starting at $f(z_{0})\in \Omega'$ and stopped at its first exist $T'$ of $\Omega'$. 
\end{proposition}

\begin{figure}[tb]
\includegraphics[scale=0.64]{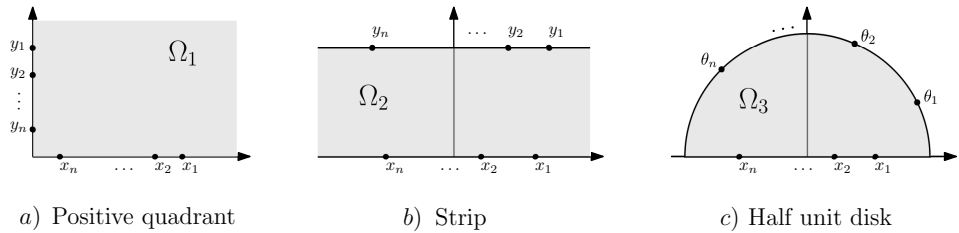}
\caption{Simply connected domains in the complex plane $\mathbb{C}$.}
\label{dom}
\end{figure}

These properties ensure that, under mild conditions on $\partial\Omega$ (for example, if $\partial\Omega$ is determined by a Jordan curve; see also Section 2.3 of \cite{Lawler2} for more general conditions), we have that for all $A\subset\partial\Omega$
\begin{align}\label{quetz}
\mathbb{P}^{z_{0}}(B_{T}\in A)=\mathbb{P}^{f(z_{0})}(f(B_{T})\in f(A))=\mathbb{P}^{z_{0}'}(B'_{T'}\in f(A)),
\end{align}
where $B'$ is another two-dimensional Brownian motion. If we set $\mu_{z_{0}, \Omega}(A)=\mathbb{P}^{z_{0}}(B_{T}\in A)$, $A\subset\partial\Omega,$
then $\mu_{z_{0}, \Omega}$ defines a measure on $\partial\Omega$, which is called the {\it harmonic measure} or {\it hitting measure} on $\partial\Omega$. Therefore, identity (\ref{quetz}) becomes 
\begin{align}\label{book}
\mu_{z_{0}, \Omega}(A)=\mu_{z_{0}', \Omega'}(f(A)),\quad \text{for all}\,\, A\subset\partial\Omega.
\end{align}
If both measures are absolutely continuous with respect to one-dimensional Lebesgue measure, or {\it lenght} (which is the case in all the examples considered in this paper), then, from (\ref{book}) we obtain (see also \cite{Lawler2,KL}):

\begin{proposition}
\label{ccphd}
Let $\Omega$ and $\Omega'$ be two simply connected domains with $z_{0}\in \Omega$. Let $f:\Omega\to \Omega'$ be a conformal map and  set $z'_{0}=f(z_{0})$. Assume that we can define harmonic measures $\mu_{z_{0}, \Omega}$ and $\mu_{z_{0}',\Omega'}$ and both are absolutely continuous with respect to one-dimensional Lebesgue measure (lenght), therefore
\begin{align}\label{Poissonker}
h_{\Omega}(z_{0}, y)=|f'(y)|h_{\Omega'}(f(z_{0}),f(y)),
\end{align}
where $h_{\Omega}(z_{0},\cdot)$ and $h_{\Omega'}(z_{0}',\cdot)$ are the corresponding densities of $\mu_{z_{0}, \Omega}$ and $\mu_{z_{0}',\Omega'}$, respectively.
\end{proposition}

\begin{definition}\label{definitionkerles}
When the harmonic measure $\mu_{z_{0}, \Omega}$ has a density $h_{\Omega}(z_{0},y)$ with respect to one-dimensional Lebesgue measure (lenght), we call this density the hitting density or Poisson kernel of $\Omega$. 
\end{definition}

We often drop the suffix $\Omega$ in the definition above and simply write $h=h_{\Omega}$. In practice, the explicit computation of the harmonic measure (or its density) for an arbitrary simply connected domain $\Omega$ is not an easy task, but there are some examples where this computation can be easily performed. In Sections \ref{bmpq}, \ref{bmstrip} and \ref{unitdisk} we consider the positive quadrant $\Omega_{1}$, the infinite strip $\Omega_{2}$ and the upper half-circle $\Omega_{3}$, respectively (see Figure \ref{dom}):
\begin{align*}
\Omega_{1}&=\{z\in\mathbb{C}: \mathfrak{Re}(z)>0, \mathfrak{Im}(z)>0\},\\
\Omega_{2}&=\{z\in\mathbb{C}: 0<\mathfrak{Im}(z)< t\},\quad t>0,\\
\Omega_{3}&=\{z\in\mathbb{C}: |z|<1, \mathfrak{Im}(z)>0\}.
\end{align*}

\subsection{Non-intersecting Brownian motions and the GOE}
\label{an}
Consider a system of $n$ independent one-dimensional Brownian motions conditioned not to intersect up to a fixed time $t>0$, starting at positions $x_{n}<x_{n-1}<...<x_{1}$, respectively. This is the $n$-dimensional Brownian motion starting at $x=(x_{1},...,x_{n})\in\mathbb{R}^{n}$ and conditioned to stay in the chamber $C=\{y\in\mathbb{R}^{n}: y_{n}<y_{n-1}<...<y_{1}\}$ up to time $t>0$. Since the $n$-dimensional Brownian motion is a strong Markov process with continuous paths, the Karlin-McGregor formula \cite{Karlin} gives the (unnormalised) density of the positions of the process at time $t$:
\begin{align}\label{transition}
\hat p_{t}(x,y)=\det\left[p_{t}(x_{i},y_{j})\right]_{i,j=1}^{n},\quad x,y\in C, 
\end{align}
where $$p_{t}(x,y)=\frac{1}{\sqrt{2\pi t}}e^{-\frac{(x-y)^{2}}{2t}}.$$   
Let $M_{t,x}$ be the normalisation constant for (\ref{transition}), that is,
$$M_{t,x}=\int_{C}\hat p_{t}(x,y)dy.$$ 
We have the following:

\begin{proposition}
The positions at time $t>0$ of $n$ independent one-dimensional Brownian motions, started at the origin, and conditioned not to intersect up to time $t>0$, are given by
\begin{align*}
\lim_{\substack{x\in C\\x\to0}}\frac{1}{M_{t,x}}\hat p_{t}(x,y)=\frac{1}{M'_{t}}e^{-\frac{1}{2t}\sum_{i=1}^{n}y_{i}^{2}}\prod_{1\leq i<j\leq n}(y_{j}-y_{i}),
\end{align*}
where $M_{t}'$ is the corresponding
normalisation constant.
\end{proposition}

\begin{remark}
The above expression agrees with the joint density of the eigenvalues of an $n\times n$ GOE random matrix with variance parameter $t$ \cite{Mehta,Forrester}.

\end{remark}

\begin{proof}
A simple calculation shows that
\begin{align*}
\frac{1}{M_{t,x}}\hat p_{t}(x,y)&=\frac{1}{M_{t,x}}\frac{1}{(2\pi t)^{n/2}}e^{-\frac{1}{2t}\sum_{i=1}^{n}(x_{i}^{2}+y_{i}^{2})}\det\left(e^{\frac{1}{t}x_{i}y_{j}}\right)_{i,j=1}^{n},
\end{align*}
and, dividing both numerator and denominator by the Vandermonde determinant $$\Delta(x)=\prod_{1\leq i<j\leq n}(x_{j}-x_{i}),$$
we can use Lemma \ref{lemma} to compute the limit in the proposition as follows:
\begin{align*}
\lim_{\substack{x\in C\\x\to0}}\frac{1}{M_{t,x}}\hat p_{t}(x,y)&=\frac{1}{M'_{t}}t^{\frac{(n-1)n}{2}}e^{-\frac{1}{2t}\sum_{i=1}^{n}y_{i}^{2}}\det\left(\left(\frac{y_{j}}{t}\right)^{i-1}\right)_{i,j=1}^{n}\\
&=\frac{1}{M'_{t}}e^{-\frac{1}{2t}\sum_{i=1}^{n}y_{i}^{2}}\det\left(y_{j}^{i-1}\right)_{i,j=1}^{n}\\
&=\frac{1}{M'_{t}}e^{-\frac{1}{2t}\sum_{i=1}^{n}y_{i}^{2}}\prod_{1\leq i<j\leq n}(y_{j}-y_{i}),
\end{align*}
where $$M'_{t}=\int_{C}e^{-\frac{1}{2t}\sum_{i=1}^{n}y_{i}^{2}}\prod_{1\leq i<j\leq n}(y_{j}-y_{i})dy.$$ 

\end{proof}

\subsection{Brownian motion in the positive quadrant} 
\label{bmpq}
Let's identify the two-dimensional Euclidean space $\mathbb{R}^{2}$ with the complex plane $\mathbb{C}$. The positive quadrant is the simply connected domain given by
\begin{align*}
\Omega_{1}&=\{z\in\mathbb{C}: \mathfrak{Re}(z)>0, \mathfrak{Im}(z)>0\}.
\end{align*}
For the two dimensional Brownian motion $B$, starting at a point $x$ in the positive $x$-axis,  the density of the first hitting point $y=iy$, $y\in\mathbb{R}$ in the $y$-axis, is given by the Cauchy density (see \cite{Durrett}, section 1.9): 
\begin{align*}
h'(x,y)=\frac{1}{\pi}\frac{x}{x^{2}+y^{2}},\quad y\in\mathbb{R}.
\end{align*}
Therefore, we can consider the two-dimensional `Brownian motion' $B'$ in the positive quadrant $\overline{\Omega}_{1}$ starting at $x$, with {\it normal reflection}
on the positive $x$-axis, and the positive $y$-axis acting as {\it absorbing boundary}. The process $B'$ first hits the positive $y$-axis at a point $y=iy$, $y>0$, with density
\begin{align}\label{pkq}
h(x,y)=h'(x,y)+h'(x,-y)=\frac{2}{\pi}\frac{x}{x^{2}+y^{2}},\quad y>0.
\end{align}
Consider the determinant of hitting densities
$$H(x,y)=\det\left(h(x_{i},y_{j})\right)_{i,j=1}^{n},\qquad x,y\in D,$$ 
where
$$D=\{x\in\mathbb{R}^{n}: 0<x_{n}<x_{n-1}<...<x_{1}\}.$$

\begin{proposition}\label{bmpqpp}
For any $y\in D$, and $t>0$,
\begin{align*}
\lim_{\substack{x\in D\\x\to t}}\tilde H(x,y)&=\frac{1}{M_{t}}\prod_{1\leq j\leq n}(t^{2}+y_{j}^{2})^{-n}\prod_{1\leq i<j\leq n}(y_{i}^{2}-y_{j}^{2}),
\end{align*}
where 
$$\tilde{H}(x,y)=\left(\int_{D}H(x,y)dy\right)^{-1}H(x,y),$$
and $M_t$ is the corresponding normalisation constant.
\end{proposition}

\begin{remark}
In particular, when $t=1$, the above density takes the form
 \begin{align*}
 \frac{1}{M_1}\prod_{1\leq j\leq n}(1+y_{j}^{2})^{-n}\prod_{1\leq i<j\leq n}(y_{j}^{2}-y_{i}^{2}),
 \end{align*}
 which is a Cauchy-type ensemble on the positive half-line \cite{Forrester,ForresterWitte}.
\end{remark}

\begin{proof}
For all $x,y\in D$, the function $H(x,y)$ is positive  and a Cauchy determinant (see \cite{Karlin2}). Therefore
\begin{align*}
H(x,y)=\left(\frac{2}{\pi}\right)^{n}\prod_{i=1}^{n}x_{i}\prod_{1\leq i,j\leq n}(x_{i}^{2}+y_{j}^{2})^{-1}\prod_{1\leq i<j\leq n}(x_{i}^{2}-x_{j}^{2})(y_{i}^{2}-y_{j}^{2}).
\end{align*}
Regarded as a probability density on $D$, we can consider the normalised density
\begin{align}\label{pillow}
\tilde H(x,y)=\left(\int_{D}F_{x}(y)dy\right)^{-1}F_{x}(y),
\end{align}
where $$F_{x}(y)=\prod_{1\leq i,j\leq n}(x_{i}^{2}+y_{j}^{2})^{-1}\prod_{1\leq i<j\leq n}(y_{i}^{2}-y_{j}^{2}).$$ 
Fix a real number $t>0$ and consider $\varepsilon$ such that  $0<\varepsilon<t$. We have that for all $x\in D$ such that $|x_{i}-t|<\varepsilon$, $1\leq i\leq n$, and for all $y\in D$
\begin{align*}
|F_{x}(y)|\leq\prod_{1\leq j\leq n}(T^{2}+y_{j}^{2})^{-n}\prod_{1\leq i<j\leq n}(y_{i}^{2}-y_{j}^{2}),\quad T=t-\varepsilon>0.
\end{align*}
The function on the right hand side is integrable over $D$, which can be verified by using the relation
\begin{align}\label{mapp}
e^{i\theta}=\frac{iy-\sqrt{T}}{iy+\sqrt{T}},\quad 0<\theta<\pi,
\end{align}
so that
$$\prod_{1\leq j\leq n}(T^{2}+y_{j}^{2})^{-n}\prod_{1\leq i<j\leq n}(y_{i}^{2}-y_{j}^{2})dy\propto\prod_{1\leq i<j\leq n}|e^{i\theta_{i}}-e^{i\theta_{j}}||e^{i\theta_{i}}-e^{-i\theta_{j}}|d\theta,$$
and the latter function is integrable over the bounded domain (see Section \ref{unitdisk})
$$\{\theta\in\mathbb{R}^{n}:0<\theta_{1}<\theta_{2}<...<\theta_{n}<\pi\}.$$
Therefore, by the dominated convergence theorem, when the $n$ starting points $x_{1},x_{2},...,x_{n}$ approach to the common point $t$, along the $x$-half positive axis, the limit in the proposition can be easily computed from the expression (\ref{pillow}), as
\begin{align*}
\lim_{\substack{x\in D\\x\to t}}\tilde H(x,y)&=\frac{1}{M_{t}}\prod_{1\leq j\leq n}(t^{2}+y_{j}^{2})^{-n}\prod_{1\leq i<j\leq n}(y_{i}^{2}-y_{j}^{2}),
\end{align*}
where $M_t$ is the corresponding normalisation constant. 

\end{proof}

\subsection{Brownian motion in a strip}
\label{bmstrip}
Consider the infinite strip given by
\begin{align*}
\Omega_{2}=\{z\in\mathbb{C}: 0<\mathfrak{Im}(z)< t\}.
\end{align*}
By conformal invariance of the two-dimensional Brownian motion (the function $f(z)=e^{\pi z/2t}$ maps $\Omega_{2}$ onto the positive quadrant $\Omega_{1}$), we can also consider a `Brownian motion' constrained to live in the strip $\overline{\Omega}_{2}$, starting at a point $x\in\mathbb{R}$ on the $x$-axis (which is normal reflecting) and stopped once it hits the (absorbing) boundary line $\mathfrak{Im}(z)=t$. If the process starts at a point $x\in\mathbb{R}$, and it first hits the absorbing boundary at $y:=y+it$, $y\in\mathbb{R}$, then by conformal invariance (that is, using Proposition \ref{ccphd} and formula (\ref{pkq}) above) we obtain a formula for the hitting density $h_{t}(x,y)$ of $\Omega_{2}$:
\begin{align*}
h_{t}(x,y)&=\frac{\pi}{2t}|ie^{\pi y/2t}|h(e^{\pi x/2t},e^{\pi y/2t})=\frac{1}{t}\frac{1}{e^{\pi(y-x)/2t}+e^{-\pi(y-x)/2t}}.
\end{align*}
In terms of hyperbolic functions, the above expression can be written as
\begin{align}\label{limite}
h_{t}(x,y)=\frac{1}{2t}{\rm sech}\left(\frac{\pi}{2t}(y-x)\right),\quad y\in\mathbb{R}.
\end{align}
Define the determinant of hitting densities
\begin{align}
H_{t}(x,y)=\det\left(h_{t}(x_{i},y_{j})\right)_{i,j=1}^{n},\quad x,y\in C,
\end{align} 
where 
$$C=\{x\in\mathbb{R}^{n}: x_{n}<x_{n-1}<...<x_{1}\}.$$ 

\begin{proposition}
For any $y\in C$, and $t>0$,
\begin{align*}
\lim_{\substack{x\in C\\x\to0}}\tilde H_{t}(x,y)=\frac{1}{M_t}\prod_{j=1}^{n}\sech(\frac{\pi}{2t}y_{j})\prod_{1\leq i<j\leq n}\left(\tanh(\frac{\pi}{2t}y_{i})-\tanh(\frac{\pi}{2t}y_{j})\right),
\end{align*}
where 
$$\tilde{H}_{t}(x,y)=\left(\int_{C}H_{t}(x,y)dy\right)^{-1}H_{t}(x,y),$$
$M_t$ is the corresponding normalisation constant.
\end{proposition}

\begin{proof}
For all $x,y\in C$, the function $H_{t}(x,y)$ is positive (see \cite{Karlin2}), and the following explicit expression for $H_{t}(x,y)$ can be obtained
\begin{align*}
H_{t}(x,y)=\frac{1}{(2t)^{n}}\prod_{1\leq i,j\leq n}{\rm sech}\left(\frac{\pi}{2t}(y_{j}-x_{i})\right)\prod_{1\leq i<j\leq n}{\rm sinh}\left(\frac{\pi}{2t}(x_{i}-x_{j})\right){\rm sinh}\left(\frac{\pi}{2t}(y_{i}-y_{j})\right).
\end{align*}
Consider the normalised density
$$\tilde H_{t}(x,y)=\left(\int_{C}F_{x}(t,y)dy\right)^{-1}F_{x}(t,y),$$
where
$$F_{x}(t,y)=\prod_{1\leq i,j\leq n}{\rm sech}\left(\frac{\pi}{2t}(y_{j}-x_{i})\right)\prod_{1\leq i<j\leq n}{\rm sinh}\left(\frac{\pi}{2t}(y_{i}-y_{j})\right).$$ 
Let $\varepsilon>0$. We have that for all $x\in C$ such that $|x_{i}|<\varepsilon$, $1\leq i\leq n$, and for all $y\in C$
$$|F_{x}(t,y)|\leq \prod_{j=1}^{n}\left(\frac{2c}{e^{\frac{\pi}{2t}y_{j}}+c^{-2}e^{-\frac{\pi}{2t}y_{j}}}\right)^{n}\prod_{1\leq i<j\leq n}{\rm sinh}\left(\frac{\pi}{2t}(y_{i}-y_{j})\right),$$
where $c=e^{\pi\varepsilon/2t}$. It can be verified that the function on the right hand side above is integrable over $C$. Therefore, by the dominated convergence theorem, for any $t>0$ and $y\in C$, it holds
\begin{align*}
\lim_{\substack{x\in C\\x\to0}}\tilde H_{t}(x,y)&=\frac{1}{M_t}\prod_{1\leq i,j\leq n}{\rm sech}\left(\frac{\pi}{2t}y_{j}\right)\prod_{1\leq i<j\leq n}{\rm sinh}\left(\frac{\pi}{2t}(y_{i}-y_{j})\right)\\
&=\frac{1}{M_t}\prod_{j=1}^{n}\sech(\frac{\pi}{2t}y_{j})\prod_{1\leq i<j\leq n}\left(\tanh(\frac{\pi}{2t}y_{i})-\tanh(\frac{\pi}{2t}y_{j})\right),
\end{align*}
where $M_t$ is the corresponding normalisation constant. 

\end{proof}

\subsection{Brownian motion in the half unit disk}
\label{unitdisk}
The image of the positive quadrant $\Omega_{1}$ through the conformal map $f(z)=\frac{z-1}{z+1}$ gives the upper half unit disk
\begin{align*}
\Omega_{3}=\{z\in\mathbb{C}: |z|<1, \mathfrak{Im}(z)>0 \}.
\end{align*}
The `Brownian motion' $B$ in $\overline{\Omega}_{3}$ reflects in the $x$-axis and stops once it reaches the boundary $|z|=1$. The hitting density for this process, starting at a point $x\in\mathbb{R}$, $|x|<1$, and stopped until it hits the point $z=e^{i\theta}$, $0<\theta<\pi$, is given by the well-known formula (see \cite{Durrett}, section 1.10):
\begin{align*}
h(x,\theta)&=\frac{1}{\pi}\frac{1-x^{2}}{1-2x\cos\theta +x^{2}},\quad 0<\theta<\pi.
\end{align*}
As before, consider the determinant of hitting densities
$$H(x,\theta)=\det(h(x_{i},y_{j}))_{i,j=1}^{n},\quad x\in N, \theta\in \Theta,$$
where
\begin{align*}
N=&\{x\in\mathbb{R}^{n}
: -1<x_{n}<x_{n-1}<...<x_{1}<1\}\quad{\rm and}\\
&\Theta=\{\theta\in\mathbb{R}^{n}:0<\theta_{1}<\theta_{2}<...<\theta_{n}<\pi\}.
\end{align*}

\begin{proposition}\label{bmithud}
For any $\theta\in\Theta$
\begin{align*}
\lim_{\substack{x\in N\\x\to 0}}\tilde H(x,\theta)=\frac{1}{M}\prod_{1\leq i<j\leq n}|e^{i\theta_{i}}-e^{i\theta_{j}}||e^{i\theta_{i}}-e^{-i\theta_{j}}|,
\end{align*}
where 
$$\tilde{H}(x,\theta)=\left(\int_{\Theta}H(x,\theta)d\theta\right)^{-1}H(x,\theta),$$
$M$ is the corresponding normalisation constant.
\end{proposition}

\begin{remark}
The above density can be thought of as the $\beta=1$ version of the eigenvalue density of a random matrix in $SO(2n)$, which is the subgroup of unitary matrices consisting of $2n\times2n$ orthogonal matrices with determinant one (see \cite{Conrey}).
\end{remark}

\begin{proof}
The function $H(x,\theta)$ can be expressed explicitly as
\begin{align*}
H(x,\theta)=H'(x)\prod_{1\leq i,j\leq n}(1-2x_{i}\cos\theta_{j}+x_{i}^{2})^{-1}\prod_{1\leq i<j\leq n}2(\cos\theta_{i}-\cos\theta_{j}),
\end{align*}
where
$$H'(x)=\frac{1}{\pi^{n}}\prod_{i=1}^{n}(1-x_{i}^{2})\prod_{1\leq i<j\leq n}(x_{i}-x_{j})(1-x_{i}x_{j}).$$
For all $x\in N$, $\theta\in\Theta$, the function $H_{t}(x,y)$ is positive. Consider the normalised density
$$\tilde{H}(x,\theta)=\left(\int_{\Theta}F_{x}(\theta)d\theta\right)^{-1}F_{x}(\theta),$$
where
$$F_{x}(\theta)=\frac{H(x,\theta)}{H'(x)}=\prod_{1\leq i,j\leq n}(1-2x_{i}\cos\theta_{j}+x_{i}^{2})^{-1}\prod_{1\leq i<j\leq n}2(\cos\theta_{i}-\cos\theta_{j}).$$
Let $0<\varepsilon<1$, and assume that $|x_{i}|<\varepsilon$, for all $1\leq i\leq n$. We have that $|1-2x_{i}\cos\theta_{j}+x_{i}^{2}|>(1-\varepsilon)^{2}$ and therefore
\begin{align*}
|F_{x}(\theta)|\leq 2^{n(n-1)}(1-\varepsilon)^{-2n^{2}},\quad\text{for all}\,\,\theta\in\Theta.
\end{align*}
Since $\Theta$ is a bounded set, by the bounded convergence theorem it follows that for any $\theta\in\Theta$
\begin{align}\label{coslimit}
\lim_{\substack{x\in N\\x\to 0}}\tilde H(x,\theta)
&=\frac{1}{M}\prod_{1\leq i<j\leq n}2(\cos\theta_{i}-\cos\theta_{j})\\
&=\frac{1}{M}\prod_{1\leq i<j\leq n}|e^{i\theta_{i}}-e^{i\theta_{j}}||e^{i\theta_{i}}-e^{-i\theta_{j}}|,\notag
\end{align}
where 
$$M=\int_{\Theta}\prod_{1\leq i<j\leq n}2(\cos\theta_{i}-\cos\theta_{j})d\theta$$ is the normalisation constant.

\end{proof}

\subsection{A note on excursion Poisson kernel determinants}
\label{nnoepkd}
In all the examples of Sections \ref{bmpq}, \ref{bmstrip} and \ref{unitdisk}, we have imposed both {\it absorbing} and {\it normal reflecting} boundary conditions on the domains under consideration. If, on the other hand, the whole boundary $\partial\Omega$ is absorbing, then we require a different notion of hitting density $h(x,y)$ (since the paths need to `walk' into the interior $\Omega^{\circ}=\Omega\setminus\partial\Omega$ before reaching their destination). Therefore, in order to study determinants of the form (\ref{Poissonkernel4}) and (\ref{Poissonkernel24}), we consider the so-called {\it excursion Poisson kernel} $h_{\partial\Omega}(x,y)$, which can be defined as the limit
\begin{align*}
h_{\partial\Omega}(x,y)=\lim_{\varepsilon\to0}\frac{1}{\varepsilon}h(x+\varepsilon{\bf n}_{x},y),\quad x,y\in\partial\Omega,
\end{align*}
where $h(z,y)$, $z\in\Omega$, is the usual hitting density (Definition \ref{definitionkerles}), and ${\bf n}_{x}$ is the unit normal at $x$ pointing into $\Omega$ (see \cite{Lawler2} for details). As we said before, intuitively, the excursion Poisson kernel requires the path to `walk' into $\Omega$ before reaching $\partial\Omega$, and it is the scaling limit of simple random walk excursion probabilities \cite{Lawler2,KL}.

It can be shown that, similarly to Proposition \ref{ccphd}, the excursion Poisson kernel satisfies a {\it conformal covariance property}:
\begin{align*}
h_{\partial\Omega}(x,y)=|f'(x)||f'(y)|h_{\partial\Omega'}(f(x),f(y)),
\end{align*} 
where $f:\Omega\to\Omega'$ is any conformal transformation. This implies that the determinant of excursion Poisson kernels:
\begin{align}\label{depk11}
\frac{\det(h_{\partial\Omega}(x_{i},y_{j}))_{i,j=1}^{n}}{\prod_{i=1}^{n}h_{\partial\Omega}(x_{i},y_{i})},
\end{align}
is a conformal {invariant} (see \cite{KL}). In particular, if $\Omega$ is the half unit circle of Section \ref{unitdisk}, standard calculations show that the excursion Poisson kernel is given by
\begin{align*}
h_{\partial\Omega}(x,\theta)=\frac{2}{\pi}\frac{(1-x^{2})\sin\theta}{(1-2x\cos\theta+x^{2})^{2}},\quad0<\theta<\pi,
\end{align*}
for $x\in\mathbb{R}$, $|x|<1$. The next proposition is the excursion Poisson kernel analogue of Proposition \ref{bmithud}.

\begin{proposition}\label{prepk}
As in Proposition \ref{bmithud}, let $\Theta$ be the set
$$\Theta=\{\theta\in\mathbb{R}^{n}:0<\theta_{1}<\theta_{2}<...<\theta_{n}<\pi\}.$$
Then
\begin{align*}
\lim_{x_{1},...,x_{n}\to0}\frac{\det(h_{\partial\Omega}(x_{i},\theta_{j}))_{i,j=1}^{n}}{\int_{\Theta}\det(h_{\partial\Omega}(x_{i},\theta_{j}))_{i,j=1}^{n}d\theta}=\frac{1}{M}\prod_{j=1}^{n}\sin\theta_{j}\prod_{1\leq i<j\leq n}(\cos\theta_{i}-\cos\theta_{j}),
\end{align*}
where the limit is taken over points $-1<x_{n}<x_{n-1}<...<x_{1}<1$ and $M$ is the corresponding normalisation constant.
\end{proposition}

\begin{proof}
Note that
\begin{align*}
\det(h_{\partial\Omega}(x_{i},\theta_{j}))_{i,j=1}^{n}=\left(\frac{2}{\pi}\right)^{n}\prod_{i=1}^{n}(1-x_{i}^{2})\prod_{j=1}^{n}\sin\theta_{j}\,\det\left(B\right),
\end{align*}
where $B=(b_{i,j})$ is the $n\times n$ matrix with positive entries
\begin{align*}
b_{i,j}=\frac{1}{(1-2x_{i}\cos\theta_{j}+x_{i}^{2})^{2}}.
\end{align*}
The determinant $\det(B)$ can be expressed as the product (see  \cite{CL}):
\begin{align}\label{carlitz}
\det(B)=\det\left(\frac{1}{1-2x_{i}\cos\theta_{j}+x_{i}^{2}}\right)_{i,j=1}^{n}{\rm per}\left(\frac{1}{1-2x_{i}\cos\theta_{j}+x_{i}^{2}}\right)_{i,j}^{n},
\end{align}
where the permanent of a square matrix is defined as
\begin{align*}
{\rm per}(a_{i,j})_{i,j=1}^{n}=\sum_{\sigma\in S_{n}}\prod_{i=1}^{n}a_{i,\sigma(i)}.
\end{align*}
The determinant in the right hand side of (\ref{carlitz}) was considered in Section \ref{unitdisk}. Therefore, we can conclude that  
\begin{align}\label{fullexpr}
\det(h_{\partial\Omega}(x_{i},\theta_{j}))_{i,j=1}^{n}=G(x)P(x,\theta)\prod_{j=1}^{n}\sin\theta_{j}\prod_{1\leq i<j\leq n}2(\cos\theta_{i}-\cos\theta_{j}),
\end{align}
where $G(x)$ and $P(x,\theta)$ are given by
\begin{align*}
G(x)&=\left(\frac{2}{\pi}\right)^{n}\prod_{i=1}^{n}(1-x_{i}^{2})\prod_{1\leq i<j\leq n}(x_{i}-x_{j})(1-x_{i}x_{j}),\\
P(x,\theta)&=\prod_{1\leq i,j\leq n}(1-2x_{i}\cos\theta_{j}+x_{i}^{2})^{-1}{\rm per}\left(\frac{1}{1-2x_{i}\cos\theta_{j}+x_{i}^{2}}\right)_{i,j}^{n}.
\end{align*}
For all $-1<x_{n}<x_{n-1}<...<x_{1}<1$ and $\theta\in\Theta$, the determinant (\ref{fullexpr}) is then positive and, since the term $G(x)$ depends only on the variables $x_{1},...,x_{n}$, it holds that
\begin{align*}
\frac{\det(h_{\partial\Omega}(x_{i},\theta_{j}))_{i,j=1}^{n}}{\int_{\Theta}\det(h_{\partial\Omega}(x_{i},\theta_{j}))_{i,j=1}^{n}d\theta}=\left(\int_{\Theta}Q_{x}(\theta)d\theta\right)^{-1}Q_{x}(\theta),
\end{align*}
where
\begin{align*}
Q_{x}(\theta)=P(x,\theta)\prod_{j=1}^{n}\sin\theta_{j}\prod_{1\leq i<j\leq n}2(\cos\theta_{i}-\cos\theta_{j}).
\end{align*}
Finally, note that for each $\theta\in\Theta$, $\lim P(x,\theta)=n!$ when $x_{i}\to0$, $1\leq i\leq n$, and
\begin{align*}
|Q_{x}(\theta)|\leq n!\,2^{n(n-1)}(1-\varepsilon)^{-2(n^{2}+n)},\quad\text{for all}\,\,\theta\in\Theta,
\end{align*}
whenever $|x_{i}|<\varepsilon$, $0<\varepsilon<1$, for all $1\leq i\leq n$. Since $\Theta$ is bounded, the desired result follows from the bounded convergence theorem.

\end{proof}

Proposition \ref{prepk} agrees with certain asymptotics of an excursion Poisson kernel determinant in \cite{Katori}, in the context of rectangular domains of the complex plane.

\section{Circular ensembles}
\label{circular}

In this section we consider limits of determinants of hitting densities of the (affine) form (\ref{Poissonkernel24})  
\begin{align}\label{Poissonkernel25}
H(x,y)=\det\left(\sum_{k\in\mathbb{Z}}\zeta^{k}h(x_{i},y_{j}+mk)\right)_{i,j=1}^{n}dy_{1}\cdots dy_{n},
\end{align}
where  
$$
\zeta= \left\{
        \begin{array}{ll}
            1 & \quad \text{if $n$ is odd} \\
            -1 & \quad \text{if $n$ is even},
        \end{array}
    \right.
$$
and reveal some natural connections with circular ensembles of random matrix theory, similar to the connections described in Section \ref{rmt} with Cauchy type ensembles. In particular, by considering the hitting density of the two-dimensional Brownian motion in an annulus on the complex plane, we obtain a novel interpretation of the Circular Orthogonal Ensemble (COE) (see Section \ref{secannulus}). Another example  is given in Section \ref{bmuc}, where we  review the well-known model of $n$ non-intersecting (one-dimensional) Brownian motions on the circle \cite{Werner} and detail its connection with the Circular Orthogonal Ensemble. An interesting consequence is Proposition \ref{indistinguishable}, which recovers the Karlin-McGregor (for $n$ odd) and Liechty-Wang (for $n$ even) determinant 
 formulas \cite{Karlin,Liechty}, for the transition density of $n$ {\em indistinguishable} non-intersecting Brownian motions on the circle, from the one in \cite{Werner}.

\subsection{Brownian motion on the unit circle}
\label{bmuc}
As a warm up before Section \ref{secannulus}, we describe the model of $n$ non-intersecting Brownian motions on the unit circle, originally studied by Hobson and Werner in \cite{Werner}. Here, the Brownian motions on $\mathbb{T}=\{e^{i\theta}:-\pi\leq\theta<\pi\}$ are given by 
\begin{align*}
\beta_{k}:=e^{i B_{k}},\quad 1\leq k\leq n,
\end{align*}
where $B_{1},B_{2},...,B_{n}$ are $n$ independent one-dimensional Brownian motions and we assume $n\geq2$. The following proposition shows that the above model can be studied by considering the exit time of the $n$-dimensional Brownian motion $B=(B_{1},B_{2},...,B_{n})$ of the domain 
\begin{align*}
\tilde{A}_{n}:=\{\nu\in\mathbb{R}^{n} : \nu_{n}<\nu_{n-1}<...<\nu_{2}<\nu_{1}<\nu_{n}+2\pi\}.
\end{align*}

\begin{proposition}[Hobson-Werner]
Let $B$ and $\tilde{A}_{n}$ as above. The transition density of the Brownian motion $B$ killed at its first exit from $\tilde{A}_{n}$ is given by
\begin{align}\label{cir}
q_{t}(\theta, \nu)=\sum_{\sigma\in S_{n}}\sum_{k_{1}+k_{2}+...+k_{n}=0}{\rm sgn}(\sigma)\prod_{i=1}^{n}p_{t}(\theta_{i},\nu_{\sigma(i)}+2\pi k_{i}),\quad t>0,
\end{align}
where $\theta=(\theta_{1},...,\theta_{n})\in\tilde{A}_{n}$, $\nu=(\nu_{1},...,\nu_{n})\in\tilde{A}_{n}$, and 
$$p_{t}(x,y)=\frac{1}{\sqrt{2\pi t}}e^{-\frac{(x-y)^{2}}{2t}}$$ is the normal density with mean $x$ and variance $t$.
\end{proposition}

The method of proof of the last proposition is by a path-switching argument, similar to the one of Theorem \ref{Fominaff}. The following corollary is a restatement of part $(i)$ of the main theorem in \cite{Werner} and describes the transition density for $n$ {\it labelled} particles in Brownian motion on the circle, constrained not to intersect until a fixed positive time.

\begin{cor}\label{zxc}
The (unnormalised) transition density of $n$ non-intersecting Brownian motions $(\beta_{1},...,\beta_{n})$ on the circle  is
\begin{align*}
q^{*}_{t}(e^{i\theta},e^{i\nu})=\sum_{\sigma\in S_{n}}\sum_{\substack{k_{1}+k_{2}+...+k_{n}=0\\{\rm mod}\,n}}{\rm sgn}(\sigma)\prod_{i=1}^{n}p_{t}(\theta_{i},\nu_{\sigma(i)}+2\pi k_{i}),\quad t>0
\end{align*}
where $e^{i\theta}=(e^{i\theta_{1}},...,e^{i\theta_{n}})\in\mathbb{T}^{n}$, $e^{i\nu}=(e^{i\nu_{1}},...,e^{i\nu_{n}})\in\mathbb{T}^{n}$, and
$$\theta,\nu\in C=\tilde{A}_{n}\cap\{\nu\in\mathbb{R}^{n}: -\pi\le \nu_{n}< \pi\}.$$
Moreover, $q^{*}_{t}(e^{i\theta},e^{i\nu})$ can be expressed as the sum of $n$ determinants:
\begin{align}\label{sumdet}
q^{*}_{t}(e^{i\theta},e^{i\nu})=\frac{1}{n}\sum_{u=0}^{n-1}\det(\sum_{k\in\mathbb{Z}}\eta^{uk}p_{t}(\theta_{i},\nu_{j}+2\pi k))_{i,j=1}^{n}.
\end{align}
\end{cor}
\begin{proof}
Since any point in the circle is the projection of an infinite set of points in the real line modulo $2\pi$, the first part follows immediately by summing up in (\ref{cir}) over all the images of $\nu=(\nu_{1},...,\nu_{n})\in\tilde{A}_{n}$ under translations of $2\pi$, that is
$$q^{*}_{t}(e^{i\theta},e^{i\nu})=\sum_{\ell\in\mathbb{Z}}q_{t}(\theta,\nu+2\pi\ell(1,...,1)).$$
For the second part, if $\eta=e^{i\frac{2\pi}{n}}$ is a complex root of unity, we can eliminate the condition $k_1+k_2+...+k_n=0$, mod $n$, by using the identity
\begin{align*}
\frac{1}{n}\sum_{u=0}^{n-1}\eta^{u\sum_{i=1}^{n} k_{i}}=\left\{
	\begin{array}{ll}
		1  & \mbox{if } \sum_{i=1}^{n}k_{i}=0,\,\,\text{mod}\,n \\
		0 & \mbox{otherwise},
	\end{array}
\right.
\end{align*}
and (\ref{sumdet}) follows.
\end{proof}

Interestingly, if we do not label the $n$ Brownian particles in Corollary \ref{zxc} (and therefore the locations at time $t>0$ are given by any of the $n$ cyclic permutations of the vector $(e^{i\nu_{1}},...,e^{i\nu_{n}})$ along the circle), then the corresponding transition density becomes a single determinant:

\begin{proposition}\label{indistinguishable}
The (unnormalised) transition density of $n$ `indistinguishable' non-intersecting Brownian motions on the circle is given by
\begin{align}\label{singledet}
H_{t}(e^{i\theta},e^{i\nu})=\det(\sum_{k\in\mathbb{Z}}e^{i2\pi xk} p_{t}(\theta_{i},\nu_{j}+2\pi k))_{i,j=1}^{n},\quad \theta,\nu\in C,
\end{align}
where 
$$
x= \left\{
        \begin{array}{ll}
            0 & \quad \text{if $n$ is odd}, \\
            \frac{1}{2} & \quad \text{if $n$ is even}.
        \end{array}
    \right.
$$

\end{proposition}

\begin{remark}
In particular, Proposition \ref{indistinguishable} recovers the Karlin-McGregor (for $n$ odd) and Liechty-Wang (for $n$ even) determinant 
 formulas \cite{Karlin,Liechty}, for the transition density of $n$ {\em indistinguishable} non-intersecting Brownian motions on the circle.
\end{remark}

\begin{remark}
Using modular transformations for Jacobi theta functions, the entries of the matrix in (\ref{singledet}) can be written 
in terms of theta functions as follows:
$$
 \left\{
        \begin{array}{ll}
            \frac{1}{2\pi}\,\theta_{3}\left(-\frac{(\nu_{j}-\theta_{i})}{2},e^{-t/2}\right) & \quad \text{if $n$ is odd}, \\
           \frac{1}{2\pi}\,\theta_{4}\left(-\frac{(\nu_{j}-\theta_{i})}{2},e^{-t/2}\right)  & \quad \text{if $n$ is even},
        \end{array}
    \right.
$$
where $\theta_{k}(z,q)$ is the $k$-Jacobi theta function, $z\in\mathbb{C}$, $|q|<1$ (see \cite{DLMF}).
\end{remark}

\begin{proof}[Proof of Proposition \ref{indistinguishable}]
The method of proof is by summing-up, in (\ref{sumdet}), the $n$ different destinations of the labelled process of Corollary \ref{zxc}. Fix $\theta\in C$ and $\nu\in C$. If $[\ell]\in S_{n}$ is the shift by $\ell=0,1,...,n-1$, let $\nu_{[\ell]}$ be the unique representative of $(\nu_{[\ell](1)},...,\nu_{[\ell](n)})$ in $C$. Then, the $n$ different `cyclic permutations' of the vector $e^{i\nu}=(e^{i\nu_{1}},...,e^{i\nu_{n}})\in\mathbb{T}^{n}$ along the unit circle are given by
\begin{align*}
e^{i\nu_{[\ell]}},\quad\ell=0,1,...,n-1.
\end{align*}
With the notation of Corollary \ref{zxc}, it holds that
\begin{align*}
q_{t}^{*}(e^{i\theta},e^{i\nu_{[\ell]}})=\frac{1}{n}\sum_{u=0}^{n-1}\eta^{-\ell u}{\rm sgn}(\sigma)\det(\sum_{k\in\mathbb{Z}}\eta^{uk}p_{t}(\theta_{i},\nu_{j}+2\pi k))_{i,j=1}^{n}.
\end{align*}
Finally, following the same argument as in the proof of Proposition \ref{Fominaff3}, we obtain
\begin{align*}
\sum_{\ell=0}^{n-1}q_{t}^{*}(e^{i\theta},e^{i\nu_{[\ell]}})=\det(\sum_{k\in\mathbb{Z}}e^{i2\pi xk} p_{t}(\theta_{i},\nu_{j}+2\pi k))_{i,j=1}^{n},
\end{align*}
where 
$$
x= \left\{
        \begin{array}{ll}
            0 & \quad \text{if $n$ is odd}, \\
            \frac{1}{2} & \quad \text{if $n$ is even}.
        \end{array}
    \right.
$$

\end{proof}

Following the notation of Proposition \ref{indistinguishable}, consider now the normalised density 
$$\tilde{H}_{t}(e^{i\theta},e^{i\nu})=\frac{1}{M_{t,\theta}}H_{t}(e^{i\theta},e^{i\nu}),\quad\theta,\nu\in C,$$
where
$$M_{t,\theta}=\int_{C}H_{t}(e^{i\theta},e^{i\nu})d\nu .$$
The following proposition is essentially a reformulation of $(ii)$ and $(iii)$ of the main theorem in \cite{Werner}, here stated in the case of $n$ {\it indistinguishable} non-intersecting Brownian motions on the circle. We also take into consideration the corresponding normalisation constants.

\begin{proposition}\label{jac}
For any $\theta,\nu\in C$,
\begin{align}\label{golon}
\lim_{t\to\infty}\tilde{H}_{t}(e^{i\theta},e^{i\nu})&=\frac{1}{M}\prod_{1\leq i<j\leq n}|e^{i\nu_{j}}-e^{i\nu_{i}}|,
\end{align}
where 
$$\lim_{t\to\infty}M_{t,\theta}=M,$$
and $M$ is the corresponding normalisation constant in the right hand side of (\ref{golon}).
\end{proposition}

\begin{remark}
The above limit agrees with the eigenvalue density of the Circular Orthogonal 
Ensemble (COE), defined on $C=\tilde{A}_{n}\cap\{\nu\in\mathbb{R}^{n}: -\pi\le \nu_{n}< \pi\}$.
\end{remark}

\begin{proof}[Proof of Proposition \ref{jac}]
Using the Poisson summation formula for each entry of the matrix array in (\ref{singledet}), we have
\begin{align*}
\sum_{k\in\mathbb{Z}}e^{i2\pi xk} p_{t}(\theta_{i},\nu_{j}+2\pi k)=\frac{1}{2\pi}\sum_{k\in\mathbb{Z}}e^{-i(\nu_{j}-\theta_{i})(x+k)}e^{-\frac{t}{2}(x+k)^{2}}.
\end{align*}
Alternatively, the above can be seen as a direct consequence of the definitions by infinite series of the Jacobi's theta functions $\theta_{3}$ and $\theta_{4}$ (see second remark after Proposition \ref{indistinguishable}). Now, by standard properties of determinants we obtain
\begin{align*}
H_{t}(e^{i\theta},e^{i\nu})&=\frac{1}{(2\pi)^{n}}\sum_{\substack{k_{1}<k_{2}<...<k_{n}\\k_{i}\in\mathbb{Z}}}\det(e^{-i\nu_{j}(x+k_{i})})\det(e^{i\theta_{j}(x+k_{i})})g_{t}({\bf k}),
\end{align*}
where ${\bf k}=(k_{1},...,k_{n})$ and
\begin{align*}
g_{t}({\bf k})=\exp(-\frac{t}{2}\sum_{i=1}^{n}(x+k_{i})^{2}).
\end{align*}
Remember from (\ref{singledet}) that $x=0$ if $n$ is odd and $x=1/2$ if $n$ is even. Regarding the term $g_{t}({\bf k})$, note that over all sequences of integers $k_{1}<k_{2}<...<k_{n}$, we have
\begin{align*}
\min_{\substack{k_{i}\in\mathbb{Z}\\k_{1}<...<k_{n}}} \frac{1}{2}\sum_{i=1}^{n}(x+k_{i})^{2}=\frac{n(n-1)(n+1)}{24},
\end{align*}
and the minimum is attained uniquely at $k_{i}=k'_{i}$, $1\leq i\leq n$, where
\begin{align*}
x+k'_{i}=i-\frac{n+1}{2},\quad i=1,2,...,n.
\end{align*}
Therefore
\begin{align*}
H_{t}(e^{i\theta},e^{i\nu})=\frac{g_{t}({\bf k}')}{(2\pi)^{n}}\left( \det(e^{i\theta_{j}(x+k'_{i})})_{i,j=1}^{n} \det(e^{-i\nu_{j}(x+k'_{i})})_{i,j=1}^{n}+Q_{t}(\theta,\nu)\right),
\end{align*}
where $Q_{t}(\theta,\nu)$ satisfies 
\begin{align*}
|Q_{t}(\theta,\nu)|&\leq (n!)^2 \sum_{\substack{k_{1}<k_{2}<...<k_{n}\\{\bf k}\not={\bf k}'}}\frac{g_{t}({\bf k})}{g_{t}({\bf k}')}\\
&= (n!)^2 \sum_{\substack{k_{1}<k_{2}<...<k_{n}\\{\bf k}\not={\bf k}'}}e^{-t\left(\frac{1}{2}\sum_{i=1}^{n}(x+k_{i})^{2}-\frac{n(n-1)(n+1)}{24}\right)}.
\end{align*}
It is not difficult to check that $Q_{t}=o(1)$ uniformly in $\theta$ and $\nu$, as $t\to\infty$.  Furthermore, the normalised density $\tilde{H}_{t}(e^{i\theta},e^{i\nu})$ can be written as
\begin{align*}
\tilde{H}_{t}(e^{i\theta},e^{i\nu})=\left(\int_{C}F_{t}^{\theta}(\nu)d\nu\right)^{-1}F_{t}^{\theta}(\nu),
\end{align*}
where
\begin{align*}
F_{t}^{\theta}(\nu)=\det(e^{i\theta_{j}(x+k'_{i})})_{i,j=1}^{n} \det(e^{-i\nu_{j}(x+k'_{i})})_{i,j=1}^{n}+Q_{t}(\theta,\nu).
\end{align*}
For each fixed $\theta\in C$, 
$$\lim_{t\to\infty}F_{t}^{\theta}(\nu)=\det(e^{i\theta_{j}(x+k'_{i})})_{i,j=1}^{n} \det(e^{-i\nu_{j}(x+k'_{i})})_{i,j=1}^{n},\quad\forall\,\nu\in C,$$
and, moreover, for all $t>T$, $T>0$ sufficiently large, we have
$$|F_{t}^{\theta}(\nu)|\leq (n!)^{2}(1+K),\quad\forall\,\nu\in C,$$
where $K$ is a positive constant. Therefore, by the bounded convergence theorem, for any $\theta,\nu\in C$, it holds that
\begin{align*}
\lim_{t\to\infty}\tilde{H}_{t}(e^{i\theta},e^{i\nu})&=\frac{1}{\int_{C}\det(e^{-i\nu_{j}(x+k'_{i})})d\nu}\det(e^{-i\nu_{j}(x+k'_{i})})_{i,j=1}^{n}\\
 &=\frac{1}{M}\prod_{1\leq i<j\leq n}|e^{i\nu_{j}}-e^{i\nu_{i}}|,
\end{align*}
where $M$ is the corresponding normalisation. For the last equality, see \cite[p208]{Mehta}.

\end{proof}

\subsection{Brownian motion in an annulus.}
\label{secannulus}
Let $0<r<1$ and $\Omega$ be the annulus centered at the origin defined by
\begin{align*}
\Omega=\{z\in\mathbb{C}: r<|z|< 1\}.
\end{align*}
Consider the `Brownian motion' $B$ in $\overline{\Omega}$, with normal reflection on the inner circle 
(of radius $r$), and stopped once it first hits the unit circle. The conformal invariance of the two-dimensional Brownian motion allows us to see the trajectories of $B$ as the conformal image of a `Brownian motion' $\beta$ in the horizontal strip
\begin{align*}
\Omega'=\{z\in\mathbb{C}: 0<\mathfrak{Im}(z)< |\log r|\},
\end{align*}
with normal reflection on the real axis and absorbing boundary $\mathfrak{Im}(z)=|\log r|$, see Figure \ref{pylopac}. From Section \ref{bmstrip}, we know that if the process $\beta$ starts at a point $\theta\in\mathbb{R}$, then the distribution of its first hitting point at $\mathfrak{Im}(z)=|\log r|$ has the density
\begin{align*}
h(\theta,\nu)=\frac{1}{2|\log r|}\sech(\frac{\pi}{2|\log r|}(\nu-\theta)),\quad \nu\in\mathbb{R}.
\end{align*} 

\begin{figure}[tb]
\includegraphics[scale=.86]{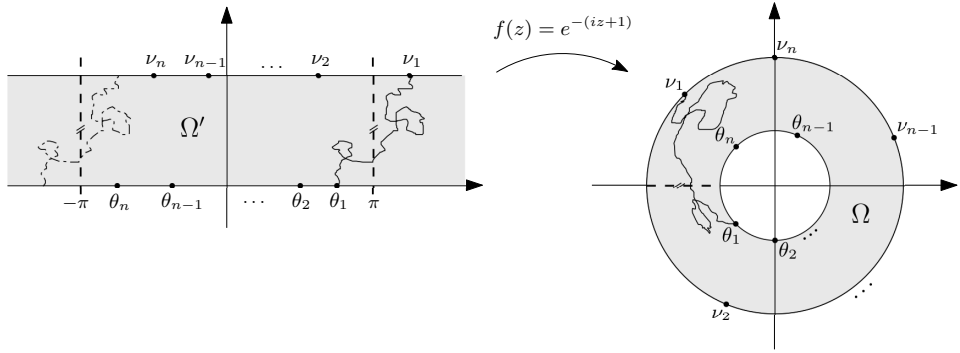}
\caption{Mapping the strip onto the annulus}
\label{pylopac}
\end{figure}

Consider the bounded set 
$$C=\tilde{A}_{n}\cap\{\nu\in\mathbb{R}^{n}: -\pi\le \nu_{n}< \pi\},$$ 
where
$$\tilde{A}_{n}:=\{\nu\in\mathbb{R}^{n} : \nu_{n}<\nu_{n-1}<...<\nu_{2}<\nu_{1}<\nu_{n}+2\pi\}.$$

\begin{definition}\label{newdefi}
Let $\eta=e^{i\frac{2\pi}{n}}$ be the $n$-th root of unity. Define, for $\theta,\nu\in C$,
\begin{align}\label{defin}
H^{a}_{r}(e^{i\theta},e^{i\nu})=\det(\sum_{k\in\mathbb{Z}}e^{i2\pi xk}h(\theta_{i},\nu_{j}+2\pi k))_{i,j=1}^{n},
\end{align}
where 
$$
x= \left\{
        \begin{array}{ll}
            0 & \quad \text{if $n$ is odd}, \\
            \frac{1}{2} & \quad \text{if $n$ is even}.
        \end{array}
    \right.
$$
\end{definition}

\begin{remark}
The strip $\Omega'\subset\mathbb{C}$ is clearly invariant under horizontal translations by $2\pi k$, $k\in\mathbb{Z}$, and therefore the determinant (\ref{defin}) is a determinant of hitting densities of the affine form (\ref{Poissonkernel24}), described at the beginning of Section \ref{rmt}. Since (\ref{defin}) is defined as a natural continuous analogue of the determinant in Proposition \ref{Fominaff3}, we expect the determinant $H_{r}^{a}(e^{i\theta},e^{i\nu})$ to be positive and be interpreted (informally) as the probability that $n$ independent trajectories $B_{1},...,B_{n}$ of the process $B$ in the annulus $\overline{\Omega}$, starting at positions
\begin{align*}
re^{i\theta_{j}},\quad j=1,...,n,
\end{align*}
will hit the unit circle by first time at points
\begin{align*}
e^{i\nu_{j}},\quad j=1,...,n,
\end{align*}
with an angle in each of the intervals $(\nu_{j},\nu_{j}+d\nu_{j})$, $j=1,...,n$, and whose trajectories are constrained to satisfy
\begin{align*}
B_{j}\cap LE(B_{j-1})=\emptyset,\quad 1< j\leq n,\quad\text{and}\quad B_{1}\cap LE(B_{n})=\emptyset.
\end{align*}
Note that we do {\it not} require that the trajectory which started at point $re^{i\theta_{j}}$ hits the unit circle at the corresponding point $e^{i\nu_{j}}$. 
\end{remark}

\begin{remark}
If $n$ is odd, the entries of the matrix in (\ref{defin}) can be written as
$$
\frac{1}{2\pi}\,\theta_{3}(0,r)\theta_{2}(0,r)\frac{\theta_{3}\left(\frac{i\pi}{2|\log r|}(\nu_{j}-\theta_{i}),r\right)}{\theta_{2}\left(\frac{i\pi}{2|\log r|}(\nu_{j}-\theta_{i}),r\right)},
$$
where $\theta_{k}(z,q)$ is the $k$-Jacobi theta function, $z\in\mathbb{C}$, $|q|<1$ (see \cite{DLMF}).
\end{remark}

Consider the normalised density
$$\tilde{H}^{a}_{r}(e^{i\theta},e^{i\nu})=\frac{1}{M_{r,\theta}}H^{a}_{r}(e^{i\theta},e^{i\nu}),$$
where
$$M_{r,\theta}=\int_{C}H^{a}_{r}(e^{i\theta},e^{i\nu})d\nu.$$
The following proposition gives the limit of $\tilde{H}^{a}_{r}(e^{i\theta},e^{i\nu})$ as the inner radius $r$ goes to zero. This models the situation where the $n$ Brownian motions start at the origin of the complex plane.

\begin{proposition}\label{COE}
For any $\theta,\nu\in C$, 
\begin{align}\label{finall}
\lim_{r\to0}\tilde{H}^{a}_{r}(e^{i\theta},e^{i\nu})=\frac{1}{M}\prod_{1\leq i<j\leq n}|e^{i\nu_{i}}-e^{i\nu_{j}}|,
\end{align}
where $$\lim_{r\to0}M_{r,\theta}=M,$$ 
and $M$ is the corresponding normalisation constant in the right hand side of (\ref{finall}).
\end{proposition}

\begin{remark}
The above limit agrees with the eigenvalue density of a random matrix belonging to the Circular Orthogonal 
Ensemble (COE), defined on $C$ \cite{Mehta, Forrester}.
\end{remark}

\begin{proof}[Proof of Proposition \ref{COE}]
By Lemma \ref{copy1} and standard properties of determinants, we can express (\ref{defin}) as the sum
\begin{align*}
H^{a}_{r}(e^{i\theta},e^{i\nu})&=\frac{1}{(2\pi)^{n}}\sum_{\substack{k_{1}<k_{2}<...<k_{n}\\k_{i}\in\mathbb{Z}}}\det(e^{-i\nu_{j}(x+k_{i})})\det(e^{i\theta_{j}(x+k_{i})})g_{r}({\bf k}),
\end{align*}
where ${\bf k}=(k_{1},...,k_{n})$ and
\begin{align*}
g_{r}({\bf k})=\prod_{i=1}^{n}\sech(|\log r|(x+k_{i})).
\end{align*}
Here $x=0$ if $n$ is odd and $x=1/2$ if $n$ is even. The terms $g_{r}({\bf k})$ are always positive and 
\begin{align*}
g_{r}({\bf k})\leq 2^{n}r^{\sum_{i=1}^{n}|x+k_{i}|}.
\end{align*}
If we minimise $\sum_{i=1}^{n}|x+k_{i}|$ over all sequences of integers $k_{1}<k_{2}<...<k_{n}$, we obtain
\begin{align*}
\min_{\substack{k_{i}\in\mathbb{Z}\\k_{1}<...<k_{n}}}\sum_{i=1}^{n}|x+k_{i}|=\frac{n^{2}-[n]}{4},\quad n\equiv[n]\in\{0,1\} \,\,\,\text{mod}\,2,
\end{align*}
and the minimum is attained uniquely at $k'_{i}=k_{i}$, $1\leq i\leq n$, where
\begin{align*}
x+k'_{i}=i-\frac{n+1}{2},\quad i=1,2,...,n.
\end{align*}
Hence, the function $H^{a}_{r}(e^{i\theta},e^{i\nu})$ can be expressed as
\begin{align*}
H_{t}^{a}(e^{i\theta},e^{i\nu})=\frac{g_{r}({\bf k}')}{(2\pi)^{n}}\left(\det(e^{i\theta_{j}(x+k'_{i})})_{i,j=1}^{n}\det(e^{-i\nu_{j}(x+k'_{i})})_{i,j=1}^{n}+Q'_{t}(\theta,\nu)\right),
\end{align*}
where $Q'_{t}(\theta,\nu)$ satisfies 
\begin{align*}
|Q'_{t}(\theta,\nu)|&\leq (n!)^{2}\sum_{\substack{k_{1}<k_{2}<...<k_{n}\\{\bf k}\not={\bf k}'}}\frac{g_{r}({\bf k})}{g_{r}({\bf k}')}\\
&\leq (n!)^{2}\,2^{n}\sum_{\substack{k_{1}<k_{2}<...<k_{n}\\{\bf k}\not={\bf k}'}}r^{\left(\sum_{i=1}^{n}|x+k_{i}|-\frac{n^{2}-[n]}{4}\right)}.
\end{align*}
One can check that $Q'_{t}=o(1)$ uniformly in $\theta$ and $\nu$, as $r\to0$. 
As in the proof of Proposition \ref{jac}, the bounded convergence theorem implies that, 
for each $\theta,\nu\in C$
\begin{align*}
\lim_{r\to0}\tilde{H}_{r}^{a}(e^{i\theta},e^{i\nu})&=\frac{1}{\int_{C}\det(e^{-i\nu_{j}(x+k'_{i})})d\nu}\det(e^{-i\nu_{j}(x+k'_{i})})_{i,j=1}^{n}\\
&=\frac{1}{M}\prod_{1\leq i<j\leq n}|e^{i\nu_{j}}-e^{i\nu_{i}}|,
\end{align*}
which concludes the proof. For the last identity, see \cite[p208]{Mehta}).

\end{proof}

\section{Conclusions}

We have developed connections between loop-erased walks in two dimensions and random matrices,
based on an identity of S. Fomin~\cite{Fomin}.  This complements earlier work of Sato and Katori~\cite{Katori},
where an example of this type of connection was exhibited in a slightly different context, as explained in
Sections \ref{crmt} and \ref{nnoepkd}.
These connections resemble the well-known relations between non-intersecting processes in one 
dimension and random matrices.  For two-dimensional Brownian motions in suitable simply connected 
domains, conditioned (in an appropriate sense) to satisfy a certain non-intersection condition, 
we obtain, in particular scaling limits, eigenvalue densities of Cauchy type.  

As a first step towards the consideration of non-simply connected domains, we have formulated and proved
an affine (circular) version of Fomin's identity.  Applying this in the context of independent Brownian motions
in an annulus, conditioned to satisfy a circular version of Fomin's non-intersection condition, 
we obtain, in a particular scaling limit, the circular orthogonal ensemble of random matrix theory.

Exploring relations between random matrices, SLE and related combinatorial models,
seems to be an interesting direction for future research.
We hope that our preliminary findings will motivate further developments in this direction.

\appendix

\section{}
\label{app}

For any $n$-tuple of complex numbers ${\bf x}=(x_{1},...,x_{n})$ let
\begin{align*}
\Delta({\bf x})=\det(x_{j}^{i-1})_{i,j=1}^{n}=\prod_{1\leq i<j\leq n}(x_{j}-x_{i}).
\end{align*}

\begin{lemma}\label{lemma}
Let $h_{i}$, $1\leq i\leq n$, be functions which are (complex) analytic at $x$.
Let ${\bf x}=(x,...,x)$ and ${\bf y}=(y_1,\ldots,y_n)$.  Then
\begin{align*}
\lim_{{\bf y}\to{\bf x}}\frac{1}{\Delta({\bf y}-{\bf x})}\det(h_{j}(y_i))_{i,j=1}^{n}=C\cdot\det\left(\partial^{i-1}_{x}h_{j}(x)\right)_{i,j=1}^n,
\end{align*}
where $C=\prod_{j=1}^{n-1}\frac{1}{j!}$.
\end{lemma}

\begin{proof}
By the Weierstrass preparation theorem, it suffices to prove the statement for 
${\bf y}={\bf x}^{\varepsilon}$ with $\varepsilon\to0$, where
$${\bf x}^{\varepsilon}={\bf x}+\varepsilon\delta,\qquad\varepsilon>0,\qquad \delta=(n-1,n-2,...,0).$$
The statement now follows from arguments presented, for example, in \cite{Wenchang}.
This is given as follows.
Define the difference operator $D$ with increment $\varepsilon$ by
\begin{align*}
D^{0}h(x)=h(x),&\quad Dh(x)=h(x+\varepsilon)-h(x)\\
D^{i+1}h(x)&=D(D^{i}h(x)),\quad n\geq1.
\end{align*}
Then
\begin{align}\label{newton}
D^{i}h(x)=\sum_{k=0}^{i}(-1)^{i+k}{i\choose k}h(x+k\varepsilon),\quad i\geq0,
\end{align}
and the operators $\partial_{x}$ and $D$ are related through the identity
\begin{align*}
\lim_{\varepsilon\to0}\frac{D^{i}h(x)}{\varepsilon^{i}}=\partial^{i}_{x}h(x).
\end{align*}
Using the relation (\ref{newton}), we have the matrix decomposition:
\begin{align*}
\left[D^{i-1}h_{j}(x)\right]_{i,j=1}^{n}&=\left[\sum_{k=0}^{i-1}(-1)^{i-1+k}{i-1\choose k}h_{j}(x+k\varepsilon)\right]_{i,j=1}^{n}\\
&=\left[(-1)^{i-1+k-1}{i-1\choose k-1}\right]_{i,k=1}^{n}\left[h_{j}(x+(k-1)\varepsilon)\right]_{k,j=1}^{n}.
\end{align*}
Note that the matrix in the middle is lower triangular, so its determinant is the product of its diagonal entries and therefore
\begin{align*}
\det\left(D^{i-1}h_{j}(x)\right)_{i,j=1}^{n}=\det\left(h_{j}(x+(k-1)\varepsilon)\right)_{k,j=1}^{n}.
\end{align*}
As a consequence, we obtain the following identity:
\begin{align*}
\det(\partial_{x}^{i-1}h_{j}(x))_{i,j=1}^{n}&=\lim_{\varepsilon\to0}\varepsilon^{-{n\choose2}}\det\left(h_{j}(x+(i-1)\varepsilon)\right)_{i,j=1}^{n}.
\end{align*}
Finally, note that 
$$\Delta({\bf x}^{\varepsilon}-{\bf x})=\Delta(\varepsilon\delta)=\varepsilon^{{n\choose 2}}\Delta(\delta)=(-\varepsilon)^{{n\choose 2}}\prod_{j=1}^{n-1}j!,$$
and therefore
$$\det(\partial_{x}^{i-1}h_{j}(x))_{i,j=1}^{n}=\left(\prod_{j=1}^{n-1}j!\right)\lim_{\varepsilon\to0}\frac{1}{\Delta({\bf x}^{\varepsilon}-{\bf x})}\det\left(h_{j}(x+(n-i)\varepsilon)\right)_{i,j=1}^{n},$$
as required.

\end{proof}

\begin{lemma}\label{copy1}
Let $\eta=e^{i\frac{2\pi}{n}}$ be the $n$-th root of unity. Therefore 
\begin{align*} 
\sum_{k\in\mathbb{Z}}e^{i2\pi xk}h(\theta,\nu+2\pi k)=\frac{1}{2\pi}\sum_{k\in\mathbb{Z}}\sech\left(|\log r|(x+k)\right)e^{-i(\nu-\theta)(x+k)}.
\end{align*}
\end{lemma}

\begin{proof}
If $\eta=e^{i\frac{2\pi}{n}}$ is the $n$-th root of unity, the left-hand side above can be expressed as
\begin{align}\label{poi}
\sum_{k\in\mathbb{Z}}e^{i2\pi xk}h(\theta,\nu+2\pi k)&=\frac{1}{2|\log r|}\sum_{k\in\mathbb{Z}}e^{i2\pi xk}\hat{f}(k),
\end{align}
where $\hat{f}(k)$ is the Fourier transform of the function $$f(\xi)=\frac{|\log r|}{\pi}\sech(|\log r|\xi)e^{-i(\nu-\theta)\xi}.$$ Therefore, applying the Poisson summation formula to the right hand side of (\ref{poi}), we obtain
\begin{align*}
\frac{1}{2|\log r|}\sum_{k\in\mathbb{Z}}e^{i2\pi xk}\hat{f}(k)&=\frac{1}{2|\log r|}\sum_{k\in\mathbb{Z}}f(x+k)\\
&=\frac{1}{2\pi}\sum_{k\in\mathbb{Z}}\sech(|\log r|\left(x+k\right))e^{-i(\nu-\theta)\left(x+k\right)}.
\end{align*}

\end{proof}

\end{document}